\documentclass[a4paper]{amsart}

\usepackage[english]{babel} 
\usepackage{amsmath,amsthm,amsfonts,amssymb,graphicx,units,color}

\newcommand{\comment}[1]{}

\theoremstyle{plain}
\newtheorem{theo}{Theorem}[section]
\newtheorem{lem}[theo]{Lemma}

\newtheorem{cor}[theo]{Corollary}
\newtheorem{rem}[theo]{Remark}

\theoremstyle{definition}

\numberwithin{equation}{section}

\def\reals{\mathbb{R}}
\def\realsnotz{\mathbb{R}\setminus\{0\}}

\def\p{\partial}

\def\Amat{\alpha}
\def\L{\Lambda}

\def\elems{\mathcal{T}}
\def\elem{T}

\def\ul{\underline}
\def\ol{\overline}
\def\ubr{\underbrace}

\def\ds{\displaystyle}

\newcommand{\set}[2]{\{#1\,\mid\,#2\}}

\DeclareMathOperator{\Max}{M}
\DeclareMathOperator{\A}{A}

\def\As{\A^{*}}

\DeclareMathOperator{\ed}{d}

\def\na{\nabla}
\def\nas{\na_{\!\!\mathrm{s}}\,}
\DeclareMathOperator{\rot}{rot}

\DeclareMathOperator{\opdiv}{div}
\def\div{\opdiv}
\DeclareMathOperator{\Div}{Div}
\def\Divs{\Div_{\mathrm{s}}\,}
\DeclareMathOperator{\id}{id}
\DeclareMathOperator{\sym}{sym}

\DeclareMathOperator{\R}{Q}

\DeclareMathOperator{\sign}{sign}

\def\om{\Omega}
\def\ga{\Gamma}
\def\gad{\ga_{\mathtt{D}}}
\def\gan{\ga_{\mathtt{N}}}
\def\gar{\ga_{\mathtt{R}}}
\def\eps{\epsilon}

\def\ft{\tilde{f}}
\def\ut{\tilde{u}}
\def\pt{\tilde{p}}
\def\xt{\tilde{x}}
\def\yt{\tilde{y}}
\def\Et{\tilde{E}}
\def\Ht{\tilde{H}}

\def\alphao{\alpha_{1}}
\def\alphat{\alpha_{2}}

\def\M{\mathcal{M}}
\def\Mi{\mathcal{M}_{i}}
\def\Mrd{\M_{\mathrm{rd}}}
\def\Mrdi{\M_{i,\mathrm{rd}}}
\def\Mec{\M_{\mathrm{ec}}}
\def\Meci{\M_{i,\mathrm{ec}}}
\def\Mle{\M_{\mathrm{le}}}
\def\Mlei{\M_{i,\mathrm{le}}}

\def\R{\mathcal{I}}
\def\Ri{\mathcal{I}_{i}}
\def\Rrd{\R_{\mathrm{rd}}}
\def\Rrdi{\R_{i,\mathrm{rd}}}

\DeclareMathOperator{\hilbert}{\sf H}
\DeclareMathOperator{\cont}{\sf C}
\DeclareMathOperator{\lebesgue}{\sf L}
\DeclareMathOperator{\rotation}{\sf R}
\DeclareMathOperator{\divergence}{\sf D}

\def\hio{\hilbert_{1}}
\def\hit{\hilbert_{2}}

\newcommand{\cgen}[2]{\cont^{#1}_{#2}}

\def\cigad{\cgen{\infty}{\gad}}
\def\cigan{\cgen{\infty}{\gan}}

\newcommand{\lgen}[2]{\lebesgue^{#1}_{#2}}
\def\lt{\lgen{2}{}}
\def\li{\lgen{\infty}{}}

\newcommand{\hgen}[2]{\hilbert^{#1}_{#2}}
\def\ho{\hgen{1}{}}

\def\hogad{\hgen{1}{\gad}}
\def\hogan{\hgen{1}{\gan}}
\def\htwo{\hgen{2}{}}

\def\hoh{\hgen{\nicefrac{1}{2}}{}}

\def\hmoh{\hgen{-\nicefrac{1}{2}}{}}

\newcommand{\dgen}[2]{\divergence^{#1}_{#2}}
\def\d{\dgen{}{}}

\def\dgan{\dgen{}{\gan}}

\newcommand{\rgen}[2]{\rotation^{#1}_{#2}}
\def\r{\rgen{}{}}

\def\rgad{\rgen{}{\gad}}

\def\rgan{\rgen{}{\gan}}

\newcommand{\norm}[1]{|#1|}

\newcommand{\normhio}[1]{\norm{#1}_{\hio}}

\newcommand{\normhioao}[1]{\norm{#1}_{\hio,\alphao}}

\newcommand{\normhioaomo}[1]{\norm{#1}_{\hio,\alphao^{-1}}}

\newcommand{\normhitat}[1]{\norm{#1}_{\hit,\alphat}}

\newcommand{\normhitatmo}[1]{\norm{#1}_{\hit,\alphat^{-1}}}

\newcommand{\normdaidaoat}[1]{\norm{#1}_{D(\A),\alphao,\alphat}}

\newcommand{\normdasaomoatmo}[1]{\norm{#1}_{D(\As),\alphao^{-1},\alphat^{-1}}}

\newcommand{\normhioomao}[1]{\norm{#1}_{\hio,\norm{\omega}\alphao}}
\newcommand{\normdaidomaoat}[1]{\norm{#1}_{D(\A),\norm{\omega}\alphao,\alphat}}
\newcommand{\normhioomaomo}[1]{\norm{#1}_{\hio,(\norm{\omega}\alphao)^{-1}}}
\newcommand{\normdasomaomoatmo}[1]{\norm{#1}_{D(\As),(\norm{\omega}\alphao)^{-1},\alphat^{-1}}}

\newcommand{\tnorm}[1]{\|#1\|}	
\newcommand{\ttnorm}[1]{|\!|\!| #1 |\!|\!|}	

\newcommand{\normlt}[1]{\norm{#1}_{\lt}}	
\newcommand{\normho}[1]{\norm{#1}_{\ho}}	
\newcommand{\normd}[1]{\norm{#1}_{\d}}

\newcommand{\scp}[2]{\langle#1,#2\rangle}

\newcommand{\scphio}[2]{\scp{#1}{#2}_{\hio}}

\newcommand{\scphioao}[2]{\scp{#1}{#2}_{\hio,\alphao}}
\newcommand{\scphioaomo}[2]{\scp{#1}{#2}_{\hio,\alphao^{-1}}}
\newcommand{\scphit}[2]{\scp{#1}{#2}_{\hit}}

\newcommand{\scphitat}[2]{\scp{#1}{#2}_{\hit,\alphat}}
\newcommand{\scphitatmo}[2]{\scp{#1}{#2}_{\hit,\alphat^{-1}}}

\newcommand{\scphioomao}[2]{\scp{#1}{#2}_{\hio,\omega\alphao}}
\newcommand{\scphioomaomo}[2]{\scp{#1}{#2}_{\hio,(\omega\alphao)^{-1}}}

\newcommand{\scplt}[2]{\scp{#1}{#2}_{\lt}}

\title[\sc Functional A Posteriori Error Equalities]
{\Large\sf Functional A Posteriori Error Control for Conforming Mixed Approximations
of Coercive Problems with Lower Order Terms}
\author{Immanuel Anjam}
\author{Dirk Pauly}
\address{Fakult\"at f\"ur Mathematik,
Universit\"at Duisburg-Essen, Campus Essen, Germany}
\email[Immanuel Anjam]{immanuel.anjam@uni-due.de}
\email[Dirk Pauly]{dirk.pauly@uni-due.de}
\date{\today}
\thanks{The first author thanks Emil Aaltonen Foundation for support.}

\begin{document}

\begin{abstract}
The results of this contribution are derived in the framework of functional type
a posteriori error estimates.
The error is measured in a combined norm which takes into account both
the primal and dual variables denoted by $x$ and $y$, respectively.
Our first main result is an error equality for all equations of the class
$$\As\A x+x=f\qquad
\text{or in mixed formulation}\qquad
\As y+x=f,\quad\A x=y,$$
where $\A$ is a linear, densely defined and closed (usually a differential)
operator and $\As$ its adjoint. In order to obtain the exact global error value 
of a conforming mixed approximation one only needs the problem data and
the mixed approximation $(\xt,\yt)\in D(\A)\times D(\As)$
of the exact solution $(x,y)\in D(\A)\times D(\As)$, i.e.,
we have the \ul{e}q\ul{ualit}y
$$\norm{x-\xt}^2
+\norm{\A(x-\xt)}^2
+\norm{y-\yt}^2
+\norm{\As(y-\yt)}^2
=\M(\xt,\yt),$$
where 
$$\M(\xt,\yt)
:=\norm{f-\xt-\As\yt}^2+\norm{\yt-\A\xt}^2$$ 
contains only known data.
Our second main result is an error estimate for all
equations of the class
$$\As\A x+ix=f\qquad
\text{or in mixed formulation}\qquad
\As y+ix=f,\quad\A x=y,$$
where $i$ is the imaginary unit. For this problem we have the \ul{two-sided} \ul{estimate}
$$\frac{\sqrt{2}}{\sqrt{2}+1} \M_i(\xt,\yt)
\leq
\norm{x-\xt}^2
+\norm{\A(x-\xt)}^2
+\norm{y-\yt}^2
+\norm{\As(y-\yt)}^2
\leq 
\frac{\sqrt{2}}{\sqrt{2}-1} \M_i(\xt,\yt),$$
where 
$$\M_i(\xt,\yt)
:=\norm{f-i\xt-\As\yt}^2 + \norm{\yt-\A\xt}^2$$ 
contains only known data.
We will point out a motivation for the study of the latter problems by time discretizations 
of linear partial differential equations and we will present an extensive list of applications.
\end{abstract}



\maketitle
\tableofcontents
\newpage


\section{Introduction}

The results presented in this paper are based on the conception of functional type a posteriori error control. 
Often these type estimates are valid for any conforming approximation and contain only global constants. 
In the case of the class of problems studied in this paper the results do not contain even global constants,
just fixed numbers. 
For a detailed exposition of the theory see the books by Repin, Neittaanm\"aki, and Mali 
\cite{repinbookone,NeittaanmakiRepin2004,MaliRepinNeittaanmaki2014} 
for a more computational point of view.

In this paper we will consider only conforming approximations, 
and we will measure the error of our approximations in a combined norm, 
which includes the error of both the primal and the dual variable. 
This is especially useful for mixed methods where one calculates an approximation 
for both the primal and dual variables, see e.g. the book 
of Brezzi and Fortin \cite{brezzifortinbookone}. 
We call this approximation pair a mixed approximation.

To the best of the authors' knowledge
functional a posteriori error estimates for combined norms were first exposed in the paper 
\cite{repinsautersmolianskiaposttwosideell}, where the authors present 
two-sided estimates bounding the error by the same quantity from below 
and from above aside from basic and global Poincar\'e type constants and some special numbers. 
In \cite{repinsautersmolianskiaposttwosideell} the authors studied real valued 
elliptic problems of the type $\As\alpha\A x=f$ given in mixed formulations $\As y=f, \alpha\A x = y$.

The first class of problems we study in the paper at hand is the linear equation 
$$(\As\alphat\A+\alphao)x=f$$
presented in the mixed formulation
\begin{equation} 
\label{eq:case1}
\As y+\alphao x=f,\quad
\alphat\A x=y,
\end{equation}
where $\alphao,\alphat$ are linear, self adjoint, and positive topological isomorphisms 
on two complex Hilbert spaces $\hio$ and $\hit$, and $\A:D(\A)\subset\hio\to\hit$ is a linear, 
densely defined, and closed operator with adjoint operator $\As:D(\As)\subset\hit\to\hio$. 
Throughout this paper we will refer to the class of problems represented by \eqref{eq:case1} 
as `Case I' in section headings. Our first main result is Theorem \ref{thm:Gmain} 
and it shortly reads as the functional a posteriori error equality
\begin{align*}
&\qquad\normhioao{x-\xt}^2
+\normhitat{\A(x-\xt)}^2
+\normhitatmo{y-\yt}^2
+\normhioaomo{\As(y-\yt)}^2\\
&=\normhioaomo{f-\alphao\xt-\As\yt}^2
+\normhitatmo{\yt-\alphat\A\xt}^2
\end{align*}
being valid for any conforming mixed approximation pair $(\xt,\yt)\in D(\A)\times D(\As)$ 
of the exact solution pair $(x,y)\in D(\A)\times D(\As)$. 
In the purely real case this result can also be derived as a special case 
of the very general result \cite[(7.2.14)]{NeittaanmakiRepin2004} 
in the context of the dual variational technique. However, we proof this result here 
by elementary methods in a general Hilbert space setting. 
Our results hold then also for the complex case. 
The equality for the purely real reaction-diffusion equation 
($\A=\nabla$, $\As=-\div$), 
was found also by Cai and Zhang \cite[Remark 6.12]{caiestimators}
and has been used for error indication of the primal variable.

The second class of problems we study in this paper is the linear equation
$$(\As\alphat\A+i\omega\alphao)x=f$$
presented in the mixed formulation
\begin{equation} 
\label{eq:case2}
\As y+i\omega\alphao x=f,\quad
\alphat\A x=y,
\end{equation}
where $\omega\in\realsnotz$. Throughout this paper we will refer 
to the class of problems represented by \eqref{eq:case2} 
as `Case II' in section headings. Our second main result is 
Theorem \ref{thm:Gmain2} and it shortly reads 
as the two-sided functional a posteriori error estimate
\begin{align*}
&\qquad\frac{\sqrt{2}}{\sqrt{2}+1} 
\left(\normhioomaomo{f-i\omega\alphao\xt-\As\yt}^2
+\normhitatmo{\yt-\alphat\A\xt}^2\right)\\
&\leq\normhioomao{x-\xt}^2+\normhitat{\A(x-\xt)}^2
+\normhitatmo{y-\yt}^2+\normhioomaomo{\As(y-\yt)}^2\\
&\leq\frac{\sqrt{2}}{\sqrt{2}-1}
\left(\normhioomaomo{f-i\omega\alphao\xt-\As\yt}^2
+\normhitatmo{\yt-\alphat\A\xt}^2\right)
\end{align*}
being valid for any conforming mixed approximation pair $(\xt,\yt)\in D(\A)\times D(\As)$ 
of the exact solution pair $(x,y)\in D(\A)\times D(\As)$. 
Note that the square root of the ratio of the upper and lower bound 
is always $1+\sqrt{2} < 2.42$, 
so the estimate gives reliable information of the combined error value. 
To the authors' best knowledge this result is new.

A motivation to study these problems comes from time-dependent PDEs. 
For many problems, if the time-derivative is discretized with `finite differences',
e.g., the backward Euler scheme, then on each time-step one solves 
a static problem of the type \eqref{eq:case1}. On the other hand, 
many time-dependent problems, e.g., the eddy current problem, can be approximated 
by a series resp. sum of static complex valued problems of the kind \eqref{eq:case2} 
by using multifrequency analysis, e.g., Fourier transformation. 
We elaborate on this in Section \ref{subsec:timediscr}.

The paper is organized as follows. In Section \ref{sec:G} 
we derive our main results in an abstract Hilbert space setting 
and in Section \ref{sec:A} we show applications of the general results 
to several partial differential equations.


\section{Results in the General Setting} \label{sec:G}

In this section we derive our main results in an abstract Hilbert space setting, 
which allows for mixed boundary conditions as well as coefficients for the case, where the
underlying problem is a PDE. 

Let ${\hio}$ and ${\hit}$ be two complex Hilbert spaces 
with the inner products $\scphio{\,\cdot\,}{\,\cdot\,}$ 
and $\scphit{\,\cdot\,}{\,\cdot\,}$, respectively. 
The right hand side $f$ belongs to ${\hio}$.
Let $\A:D(\A)\subset\hio\to\hit$ be a densely defined and closed linear operator
and $\As:D(\As)\subset\hit\to\hio$ its adjoint. We note $\A^{**}=\bar{\A}=\A$ and 
\begin{align}
\label{partint}
\forall\,\varphi\in D(\A)\quad\forall\,\psi\in D(\As)\qquad
\scphit{\A\varphi}{\psi}=\scphio{\varphi}{\As\psi}.
\end{align}
Equipped with the natural graph norms $D(\A)$ and $D(\As)$ are Hilbert spaces. 
Furthermore, we introduce two linear, self adjoint, and positive topological isomorphisms 
$\alphao:{\hio}\to{\hio}$ and $\alphat:{\hit}\to{\hit}$. Especially we have
$$\exists\,c>0\quad\forall\,\varphi\in\hio\qquad 
c^{-1}\normhio{\varphi}^2\leq\scphio{\alphao\varphi}{\varphi}\leq c\normhio{\varphi}^2$$
and the corresponding holds for $\alphat$.
In case the underlying problem is a PDE, the operators $\alpha_1$ and $\alpha_2$
describe material properties, and are often called material coefficients,
giving the constitutive laws.

For any inner product and corresponding norm we introduce weighted counterparts with sub-index notation. 
As an example, for elements from ${\hio}$ we define a new inner product 
$\scphioao{\,\cdot\,}{\,\cdot\,}:=\scphio{\alphao\,\cdot\,}{\,\cdot\,}$ 
and a new induced norm $\normhioao{\,\cdot\,}$.
Note that in Section \ref{subsec:Gcase2} we slightly abuse this notation: We use
$\scphioomao{\,\cdot\,}{\,\cdot\,}=\scphio{\omega\alphao\,\cdot\,}{\,\cdot\,}$,
where $\omega\neq0$ is possibly a negative real number. 
Clearly, this sesquilinear form neither defines an inner product nor a norm,
if $\omega$ is negative.



\subsection{Case I: Error Equality for Coefficients $\alphao$ and $\alphat$}
\label{subsec:Gcase1}

Extending the sub-index notation we define for $\varphi\in D(\A)$ and $\psi\in D(\As)$ new weighted norms
on $D(\A)$, $D(\As)$ and on the product space $D(\A)\times D(\As)$ by
\begin{align*}
\normdaidaoat{\varphi}^2
&:=\normhioao{\varphi}^2
+\normhitat{\A\varphi}^2,\\
\normdasaomoatmo{\psi}^2
&:=\normhitatmo{\psi}^2
+\normhioaomo{\As\psi}^2,\\
\tnorm{(\varphi,\psi)}^2
&:=\normdaidaoat{\varphi}^2
+\normdasaomoatmo{\psi}^2.
\end{align*}

By the Lax-Milgram's lemma
(or by Riesz' representation theorem) we get immediately:

\begin{lem}
\label{laxmilgramA}
The (primal) variational problem
\begin{align}
\label{varA}
\forall\,\varphi\in D(\A)\qquad
\scphitat{\A x}{\A\varphi}+\scphioao{x}{\varphi}=\scphio{f}{\varphi}
\end{align}
admits a unique solution $x\in D(\A)$ satisfying $\normdaidaoat{x}\leq\normhioaomo{f}$. 
Moreover, $y_{x}:=\alphat\A x$ belongs to $D(\As)$ and $\As y_{x}=f-\alphao x$.
Hence, the strong and mixed formulations
\begin{align} 
\label{eq:Gstrong}
\As\alphat\A x+\alphao x&=f,\\
\label{mixedformulation}
\As y_{x}+\alphao x&=f,\quad\alphat\A x=y_{x}
\end{align}
hold with $(x,y_{x})\in D(\A)\times \big(D(\As)\cap\alphat R(\A)\big)$.
\end{lem}

To get the dual problem, we multiply the first equation of \eqref{mixedformulation}
by $\As\psi$ with $\psi\in D(\As)$ 
taking the right weighted scalar product and use $y_{x}=\alphat\A x\in D(\As)$.
We obtain
$$\scphioaomo{\As y_{x}}{\As\psi}
+\scphioaomo{\alphao x}{\As\psi}
=\scphioaomo{f}{\As\psi}.$$
Since $x\in D(\A)$,
$$\scphioaomo{\alphao x}{\As\psi}
=\scphio{x}{\As\psi}
=\scphit{\A x}{\psi}
=\scphitatmo{y_{x}}{\psi}$$
holds, and we get again by the Lax-Milgram's lemma:

\begin{lem}
\label{laxmilgramAs}
The (dual) variational problem
\begin{align}
\label{varAs}
\forall\,\psi\in D(\As)\qquad
\scphioaomo{\As y}{\As\psi}
+\scphitatmo{y}{\psi}
=\scphioaomo{f}{\As\psi}
\end{align}
admits a unique solution $y\in D(\As)$ satisfying $\normdasaomoatmo{y}\leq\normhioaomo{f}$.
Moreover, $y=y_{x}$ holds and thus $y$ even belongs to 
$D(\As)\cap\alphat R(\A)$ with $x$ and $y_{x}$ from Lemma \ref{laxmilgramA}.
Furthermore, $\alphao^{-1}(\As y-f)\in D(\A)$ with
$A\alphao^{-1}(\As y-f)=-\alphat^{-1}y$.
\end{lem}

\begin{proof}
We just have to show that $y_{x}\in D(\As)$ solves \eqref{varAs}.
But this follows directly since for all $\psi\in D(\As)$
\begin{align*}
\scphioaomo{\As y_{x}}{\As\psi}
&=-\scphio{x}{\As\psi}
+\scphioaomo{f}{\As\psi}\\
&=-\scphit{\A x}{\psi}
+\scphioaomo{f}{\As\psi}
=-\scphitatmo{y_{x}}{\psi}
+\scphioaomo{f}{\As\psi}.
\end{align*}
Hence $y_{x}=y$ and $\A^{**}=\A$ completes the proof.
\end{proof}

\begin{rem}
\label{rem:isometrygen}
We know $\normdaidaoat{x}\leq\normhioaomo{f}$ and $\normdasaomoatmo{y}\leq\normhioaomo{f}$.
It is indeed notable that 
$$\tnorm{(x,y)}=\normhioaomo{f}$$ 
holds, which follows immediately by $y=\alphat\A x$ and
\begin{align*}
\normhioaomo{f}^2
=\normhioaomo{\As\alphat\A x+\alphao x}^2
&=\normhioaomo{\As y}^2
+\normhioaomo{\alphao x}^2
+2\Re\ubr{\scphioaomo{\As\alphat\A x}{\alphao x}}_{\ds=\scphio{\As\alphat\A x}{x}}\\
&=\normhioaomo{\As y}^2
+\normhioao{x}^2
+2\ubr{\Re\scphit{\alphat\A x}{\A x}}_{\ds=\normhitat{\A x}^2}
=\tnorm{(x,y)}^2.
\end{align*}
Thus the solution operator
$$L:\hio\to D(\A)\times D(\As);f\mapsto(x,y)$$
(equipped with the proper weighed norms)
has norm $\norm{L}=1$, i.e., $L$ is an isometry.
\end{rem}

By the latter remark the combined norm on $D(\A)\times D(\As)$ yields an isometry.
This motivates the usage of the combined norm also for error estimates.
As it turns out, we even obtain error equalities.
First we show that an error equality follows directly 
from the isometry property Remark \ref{rem:isometrygen} 
if the approximation of the primal variable $x$ is regular enough.

\begin{theo}
\label{thm:Gmain_reg}
Let $(x,y) \in D(\A)\times D(\As)$ be the exact solution of \eqref{mixedformulation}. 
Let $\xt \in D(\A)$ be arbitrary and $\yt=\alphat\A\xt \in D(\As)$. 
Then, for the mixed approximation $(\xt,\yt)$ we have
\begin{align}
\label{eq:Gmain_reg1}
\tnorm{(x,y)-(\xt,\yt)}^2=\R(\xt,\yt)
\end{align}
and the normalized counterpart
\begin{align} 
\label{eq:Gmain_reg2}
\frac{\tnorm{(x,y)-(\xt,\yt)}^2}{\tnorm{(x,y)}^2}
=\frac{\R(\xt,\yt)}{\normhioaomo{f}^2}
\end{align}
hold, where
$$\R(\xt,\yt):=\normhioaomo{f-\alphao\xt-\As\yt}^2.$$
\end{theo}

\begin{proof}
Since $\xt$ is very regular, especially $\yt = \alphat\A\xt \in D(\As)$, 
the pair $(\xt,\yt)$ is the exact solution of the problem
$$\As\yt+\alphao\xt=:\ft,\quad\alphat\A\xt=\yt ,$$
i.e., we have $L(\ft) = (\xt,\yt)$. 
Then \eqref{eq:Gmain_reg1} is given directly by Remark \ref{rem:isometrygen}:
$$\tnorm{(x,y)-(\xt,\yt)}^2 = \tnorm{L(f-\ft)}^2
= \normhioaomo{f-\ft}^2 ,$$
since $L$ is linear. The estimate \eqref{eq:Gmain_reg2} 
follows by Remark \ref{rem:solutionoperatornormgen} as well.
\end{proof}

Satisfying the high regularity property required in Theorem \ref{thm:Gmain_reg} 
may not be convenient for practical calculations. 
The next result, the first main result of the paper, 
holds for less regular approximations.

\begin{theo}
\label{thm:Gmain}
Let $(x,y),(\xt,\yt)\in D(\A)\times D(\As)$ be the exact solution of \eqref{mixedformulation}
and any conforming approximation, respectively. Then
\begin{align} 
\label{eq:Gmain1}
\tnorm{(x,y)-(\xt,\yt)}^2=\M(\xt,\yt)
\end{align}
and the normalized counterpart
\begin{align} 
\label{eq:Gmain2}
\frac{\tnorm{(x,y)-(\xt,\yt)}^2}{\tnorm{(x,y)}^2}
=\frac{\M(\xt,\yt)}{\normhioaomo{f}^2} 
\end{align}
hold, where
\begin{align} 
\label{eq:Gmain}
\M(\xt,\yt):=\normhioaomo{f-\alphao\xt-\As\yt}^2
+\normhitatmo{\yt-\alphat\A\xt}^2.
\end{align}
\end{theo}

\begin{proof}
By using \eqref{eq:Gstrong} and inserting $0=\alphat\A x-y$ we get by \eqref{partint}
\begin{align*}
\M(\xt,\yt)
&=\normhioaomo{\alphao x-\alphao\xt+\As y-\As\yt}^2
+\normhitatmo{\yt-y+\alphat\A x-\alphat\A\xt}^2\\
&=\normhioao{x-\xt}^2
+\normhioaomo{\As(y-\yt)}^2
+2\Re\scphioaomo{\alphao(x-\xt)}{\As(y-\yt)}\\
&\qquad+\normhitatmo{\yt-y}^2
+\normhitat{\A(x-\xt)}^2
+2\Re\scphitatmo{\yt-y}{\alphat\A(x-\xt)}\\
&=\normdaidaoat{x-\xt}^2+\normdasaomoatmo{y-\yt}^2\\
&\qquad+2\Re\scphio{x-\xt}{\As(y-\yt)}
-2\Re\scphit{\A(x-\xt)}{y-\yt}\\
&=\tnorm{(x,y)-(\xt,\yt)}^2.
\end{align*}
\eqref{eq:Gmain2} follows by the isometry property in Remark \ref{rem:isometrygen}, completing the proof.
\end{proof}

We note that the isometry property, i.e., $\tnorm{(x,y)}=\normhioaomo{f}$,
can be seen by inserting $(\xt,\yt)=(0,0)$ into \eqref{eq:Gmain1} as well.
The result of Theorem \ref{thm:Gmain_reg} can also be seen from Theorem \ref{thm:Gmain}.

\begin{rem}
In the purely real case, where the Hilbert spaces are over $\reals$
and all objects are real valued, Theorem \ref{thm:Gmain} 
can also be deduced as a special case of \cite[(7.2.14)]{NeittaanmakiRepin2004} 
in the book of Neittaam\"aki and Repin. 
The equality for the purely real reaction-diffusion equation ($\A = \nabla, \As = -\div$), 
was found also by Cai and Zhang in \cite[Remark 6.12]{caiestimators}.
\end{rem}

\begin{cor}
\label{cor:maintheocor}
Theorem \ref{thm:Gmain} provides the well known a posteriori error estimates 
for the primal and dual problems. 
\begin{itemize}
\item[\bf(i)]
For any $\xt\in D(\A)$ it holds
$\ds\normdaidaoat{x-\xt}^2
=\min_{\psi\in D(\As)}\M(\xt,\psi)
=\M(\xt,y)$.
\item[\bf(ii)]
For any $\yt\in D(\As)$ it holds
$\ds\normdasaomoatmo{y-\yt}^2
=\min_{\varphi\in D(\A)}\M(\varphi,\yt)
=\M(x,\yt)$.
\end{itemize}
\end{cor}

\begin{proof}
We just have to estimate
$$\normdaidaoat{x-\xt}^2
\leq\tnorm{(x,y)-(\xt,\yt)}^2
=\M(\xt,\yt)$$
and note that the left hand side does not depend on $\yt\in D(\As)$.
Setting $\psi:=\yt\in D(\As)$ we get
$$\normdaidaoat{x-\xt}^2
\leq\inf_{\psi\in D(\As)}\M(\xt,\psi).$$
But for $\psi=y\in D(\As)$ we see $\M(\xt,y)=\normdaidaoat{x-\xt}^2$,
which proves (i). Analogously, we estimate
$$\normdasaomoatmo{y-\yt}^2
\leq\tnorm{(x,y)-(\xt,\yt)}^2
=\M(\xt,\yt)$$
and note that the left hand side does not depend on $\xt\in D(\A)$.
Setting $\varphi:=\xt\in D(\A)$ we get
$$\normdasaomoatmo{y-\yt}^2
\leq\inf_{\varphi\in D(\A)}\M(\varphi,\yt).$$
But for $\varphi=x\in D(\A)$ we see $\M(x,\yt)=\normdasaomoatmo{y-\yt}^2$,
which shows (ii).
\end{proof}

\begin{rem}
\label{laxmilgramAsrem}
\mbox{}
\begin{itemize}
\item[\bf(i)]
Since $y\,\bot_{\alphat^{-1}}\,N(\As)$ by \eqref{varAs} we get immediately that $y\in\alphat\ol{R(\A)}$
by the Helmholtz decomposition $\hit=N(\As)\oplus_{\alphat^{-1}}\alphat\ol{R(\A)}$.
\item[\bf(ii)] 
If $\alphao^{-1}f\in D(\A)$ we have $z:=\alphao^{-1}\As y\in D(\A)$ 
and the strong and mixed formulations of \eqref{varAs} read
\begin{align*}
\A\alphao^{-1}\As y+\alphat^{-1}y&=\A\alphao^{-1}f,\\
\A z+\alphat^{-1}y&=\A\alphao^{-1}f,\quad \alphao^{-1}\As y=z.
\end{align*}
Then for all $\varphi\in D(\A)$ we have
\begin{align*}
\scphitat{\A z}{\A\varphi}+\scphioao{z}{\varphi}
&=-\scphit{y}{\A\varphi}+\scphioao{z}{\varphi}+\scphitat{\A\alphao^{-1}f}{\A\varphi}\\
&=\scphitat{\A\alphao^{-1}f}{\A\varphi}
\end{align*}
and hence $z\in\big(D(\A)\cap\alphao^{-1} R(\As)\big)\subset D(\A)$ 
is the unique solution of this variational problem.
Furthermore, $\alphat(\A z-\A\alphao^{-1}f)\in D(\As)$ 
and $\As\alphat(\A z-\A\alphao^{-1}f)=-\alphao z$.
If $\alphat\A\alphao^{-1}f$ belongs to $D(\As)$ this yields 
$\alphat\A z\in D(\As)$ and the strong equation
$$\As\alphat\A z+\alphao z=\As\alphat\A\alphao^{-1}f.$$
\end{itemize}
\end{rem}


\subsection{Case II: Two-Sided Error Estimate for Coefficients $i\omega\alphao$ and $\alphat$}
\label{subsec:Gcase2}

In the following $\omega \in \realsnotz$. 
Using the sub-index notation we define for $\varphi\in D(\A)$ 
and $\psi\in D(\As)$ new weighted norms
on $D(\A)$, $D(\As)$ as well as on the product space $D(\A)\times D(\As)$ by
\begin{align*}
\normdaidomaoat{\varphi}^2
&=\normhioomao{\varphi}^2
+\normhitat{\A\varphi}^2,\\
\normdasomaomoatmo{\psi}^2
&=\normhitatmo{\psi}^2
+\normhioomaomo{\As\psi}^2,\\
\ttnorm{(\varphi,\psi)}^2
&:=\normdaidomaoat{\varphi}^2
+\normdasomaomoatmo{\psi}^2.
\end{align*}

By the Lax-Milgram's lemma we get immediately:

\begin{lem}
\label{laxmilgramA2}
The (primal) variational problem
\begin{align}
\label{varA2}
\forall\,\varphi\in D(\A)\qquad
\scphitat{\A x}{\A\varphi}+i\scphioomao{x}{\varphi}=\scphio{f}{\varphi}
\end{align}
admits a unique solution $x\in D(\A)$ satisfying $\normdaidomaoat{x}\leq\sqrt{2}\normhioomaomo{f}$. 
Moreover, $y_{x}:=\alphat\A x$ belongs to $D(\As)$ and $\As y_{x}=f-i\omega\alphao x$.
Hence, the strong and mixed formulations
\begin{align} 
\label{eq:Gstrong2}
\As\alphat\A x+i\omega\alphao x&=f,\\
\label{mixedformulation2}
\As y_{x}+i\omega\alphao x&=f,\quad\alphat\A x=y_{x}
\end{align}
hold with $(x,y_{x})\in D(\A)\times \big(D(\As)\cap\alphat R(\A)\big)$.
\end{lem}

To get the dual problem, we multiply the first equation of \eqref{mixedformulation2}
by $\As\psi$ with $\psi\in D(\As)$ 
taking the right weighted scalar product and use $y_{x}=\alphat\A x\in D(\As)$.
We obtain
$$\scphioomaomo{\As y_{x}}{\As\psi}
+\scphioomaomo{i\omega\alphao x}{\As\psi}
=\scphioomaomo{f}{\As\psi}.$$
Since $x\in D(\A)$
$$\scphioomaomo{i\omega\alphao x}{\As\psi}
=i\scphio{x}{\As\psi}
=i\scphit{\A x}{\psi}
=i\scphitatmo{y_{x}}{\psi}$$
holds, and we get again by the Lax-Milgram's lemma (see Lemma \ref{laxmilgramAs}):

\begin{lem}
\label{laxmilgramAs2}
The (dual) variational problem
\begin{align}
\label{varAs2}
\forall\,\psi\in D(\As)\qquad
\scphioomaomo{\As y}{\As\psi}
+i\scphitatmo{y}{\psi}
=\scphioomaomo{f}{\As\psi}
\end{align}
admits a unique solution $y\in D(\As)$ satisfying $\normdasomaomoatmo{y}\leq\sqrt{2}\normhioomaomo{f}$.
Moreover, $y=y_{x}$ holds and thus $y$ belongs to 
$D(\As)\cap\alphat R(\A)$ with $x$ and $y_{x}$ from Lemma \ref{laxmilgramA2}.
Furthermore, $(\omega\alphao)^{-1}(\As y-f)\in D(\A)$ with
$A(\omega\alphao)^{-1}(\As y-f)=-i\alphat^{-1}y$.
\end{lem}

\begin{rem}
\label{rem:solutionoperatornormgen}
We know
\begin{equation} \label{eq:solutionoperatornormgen_known}
\normdaidomaoat{x}\leq\sqrt{2}\normhioomaomo{f}
\qquad \textrm{and} \qquad
\normdasomaomoatmo{y}\leq\sqrt{2}\normhioomaomo{f}.
\end{equation}
It is indeed notable that
\begin{equation} \label{eq:solutionoperatornormgen1}
\normhioomaomo{f}^2 = \normhioomaomo{\As y}^2 + \normhioomao{x}^2
\end{equation}
and
\begin{equation} \label{eq:solutionoperatornormgen2}
\normhioomaomo{f}\leq\ttnorm{(x,y)}\leq\sqrt{2}\normhioomaomo{f}
\end{equation}
hold\footnote{The following simple example shows that the upper bound in \eqref{eq:solutionoperatornormgen2}
is sharp: Let $\hio=\hit$, $\A:=\As:=\id$, $\omega:=1$ and $\alphao:=\alphat:=1$.
Then $x=y$, $(1+i)x=f$ and $\ttnorm{(x,y)}^2=4\normhio{x}^2=2\normhio{f}^2$.}. 
The identity \eqref{eq:solutionoperatornormgen1} follows immediately by $y=\alphat\A x$ and
\begin{align*}
\normhioomaomo{f}^2
&=\normhioomaomo{\As\alphat\A x+i\omega\alphao x}^2 \\
&=\normhioomaomo{\As y}^2
+\normhioomaomo{i\omega\alphao x}^2
+2\Re\ubr{\scp{\As\alphat\A x}{i\omega\alphao x}_{\hio,(\norm{\omega}\alphao)^{-1}}}_{\ds=-i\sign{\omega}\scphio{\As\alphat\A x}{x}}\\
&=\normhioomaomo{\As y}^2
+\normhioomao{x}^2
- 2\ubr{\Re\left(i\sign{\omega}\scphit{\alphat\A x}{\A x}\right)}_{\ds=0} .
\end{align*}
The lower bound in \eqref{eq:solutionoperatornormgen2} 
follows from \eqref{eq:solutionoperatornormgen1}. 
The upper bound in \eqref{eq:solutionoperatornormgen2} is seen as follows: 
First we take \eqref{varA2} with $\varphi=x$ and \eqref{varAs2} with $\psi=y$, and obtain
\begin{align*}
\normhitat{\A x}^2 + i\omega\normhioao{x}^2 & = \scphio{f}{x} , \\
\omega^{-1}\normhioaomo{\As y}^2 + i\normhitatmo{y}^2 & = \scphioomaomo{f}{\As y} .
\end{align*}
By taking the norm of both sides we obtain
\begin{align*}
\normhitat{\A x}^4 + \norm{\omega}^2\normhioao{x}^4 & = \norm{\scphio{f}{x}}^2 , \\
\norm{\omega}^{-2}\normhioaomo{\As y}^4 + \normhitatmo{y}^4 & = \norm{\scphioomaomo{f}{\As y}}^2 ,
\end{align*}
showing
\begin{align*}
\frac{1}{\sqrt{2}}\normdaidomaoat{x}^2 & \leq \normhioomaomo{f} \normhioomao{x} , \\
\frac{1}{\sqrt{2}}\normdasomaomoatmo{y}^2 & \leq \normhioomaomo{f} \normhioomaomo{\As y} .
\end{align*}
From these we could derive the estimates \eqref{eq:solutionoperatornormgen_known} 
for $x$ and $y$ separately. Moreover, by summing up and \eqref{eq:solutionoperatornormgen1} we get
\begin{align*}
\frac{1}{\sqrt{2}} \ttnorm{(x,y)}^2
& \leq \normhioomaomo{f} \left( \normhioomao{x}+\normhioomaomo{\As y} \right) \\
& \leq \normhioomaomo{f} \sqrt{2} \sqrt{ \normhioomao{x}^2+\normhioomaomo{\As y}^2 }
= \sqrt{2} \normhioomaomo{f}^2 
\end{align*}
and we have the upper bound in \eqref{eq:solutionoperatornormgen2}. 
Thus the norm of the solution operator
$$L_i:\hio\to D(\A)\times D(\As);f\mapsto(x,y)$$
(equipped with the proper weighted norms)
satisfies $1\leq\norm{L_i}\leq\sqrt{2}$.
Hence $L_{i}$ is `almost' an isometry.
\end{rem}

The latter remark motivates the usage of the combined norm also for error estimates.
First we show that a two-sided error estimate follows directly 
from Remark \ref{rem:solutionoperatornormgen}, if the approximation 
of the primal variable $x$ is regular enough.

\begin{theo}
\label{thm:Gmain2_reg}
Let $(x,y) \in D(\A)\times D(\As)$  be the exact solution of \eqref{mixedformulation2}. 
Let $\xt \in D(\A)$ be arbitrary and $\yt=\alphat\A\xt \in D(\As)$. 
Then, for the mixed approximation $(\xt,\yt)$ we have
\begin{align}
\label{eq:Gmain2_reg1}
\Ri(\xt,\yt)
\leq \ttnorm{(x,y)-(\xt,\yt)}^2
\leq 2 \Ri(\xt,\yt)
\end{align}
and the normalized counterpart
\begin{align} 
\label{eq:Gmain2_reg2}
\frac{1}{2} \cdot \frac{\Ri(\xt,\yt)}{\normhioomaomo{f}^2}
\leq \frac{\ttnorm{(x,y)-(\xt,\yt)}^2}{\ttnorm{(x,y)}^2}
\leq 2 \frac{\Ri(\xt,\yt)}{\normhioomaomo{f}^2}
\end{align}
hold, where
$$\Ri(\xt,\yt) := \normhioomaomo{f - i\omega\alphao\xt - \As\yt}^2 .$$
\end{theo}

\begin{proof}
Since $\xt$ is very regular, especially $\yt = \alphat\A\xt \in D(\As)$, 
the pair $(\xt,\yt)$ is the exact solution of the problem
$$\As\yt + i\omega\alphao\xt=:\ft , \quad \alphat\A\xt=\yt ,$$
i.e., we have $L_i(\ft) = (\xt,\yt)$. 
Then \eqref{eq:Gmain2_reg1} is given directly by Remark \ref{rem:solutionoperatornormgen}:
$$\normhioomaomo{f-\ft}^2
\leq \ttnorm{(x,y)-(\xt,\yt)}^2 = \ttnorm{L_i(f-\ft)}^2
\leq 2\normhioomaomo{f-\ft}^2$$
The estimate \eqref{eq:Gmain2_reg2} 
follows by Remark \ref{rem:solutionoperatornormgen} as well.
\end{proof}

The square root of the ratio of the bounds in \eqref{eq:Gmain2_reg1} 
is always $\sqrt{2} < 1.42$. The square root of the ratio of the bounds 
in \eqref{eq:Gmain2_reg2} is always $2$. 
However, satisfying the high regularity property required in Theorem \ref{thm:Gmain2_reg} 
may not be convenient for practical calculations. The next result, 
the second main result of the paper, holds for less regular approximations.

\begin{theo} 
\label{thm:Gmain2}
Let $(x,y),(\xt,\yt)\in D(\A)\times D(\As)$ be the exact solution of \eqref{mixedformulation2}
and any conforming approximation, respectively. Then
\begin{align} 
\label{eq:Gmain21}
\frac{\sqrt{2}}{\sqrt{2}+1} \Mi(\xt,\yt)
\leq \ttnorm{(x,y)-(\xt,\yt)}^2 \leq
\frac{\sqrt{2}}{\sqrt{2}-1} \Mi(\xt,\yt)
\end{align}
and the normalized counterpart
\begin{align} 
\label{eq:Gmain22}
\frac{\sqrt{2}}{2(\sqrt{2}+1)} \cdot \frac{\Mi(\xt,\yt)}{\normhioomaomo{f}^2} 
\leq \frac{\ttnorm{(x,y)-(\xt,\yt)}^2}{\ttnorm{(x,y)}^2} \leq
\frac{\sqrt{2}}{\sqrt{2}-1} \cdot \frac{\Mi(\xt,\yt)}{\normhioomaomo{f}^2} 
\end{align}
hold, where
\begin{align} 
\label{eq:Gmain2_}
\Mi(\xt,\yt):=\normhioomaomo{f-i\omega\alphao\xt-\As\yt}^2
+\normhitatmo{\yt-\alphat\A\xt}^2.
\end{align}
\end{theo}

\begin{proof}
Using \eqref{eq:Gstrong2} and inserting $0=\alphat\A x-y$ we get
\begin{align}
\Mi(\xt,\yt)
&=\normhioomaomo{i\omega\alphao x-i\omega\alphao\xt+\As y-\As\yt}^2
+\normhitatmo{\yt-y+\alphat\A x-\alphat\A\xt}^2 \nonumber \\
&=\normhioomao{x-\xt}^2
+\normhioomaomo{\As(y-\yt)}^2
+2\Re\scp{i\omega\alphao(x-\xt)}{\As(y-\yt)}_{\hio,(\norm{\omega}\alphao)^{-1}} \nonumber \\
&\qquad+\normhitatmo{\yt-y}^2
+\normhitat{\A(x-\xt)}^2
+2\Re\scphitatmo{\yt-y}{\alphat\A(x-\xt)} \nonumber \\
&=\normdaidomaoat{x-\xt}^2+\normdasomaomoatmo{y-\yt}^2 \nonumber \\
&\qquad+2\sign{\omega}\,\Re\scphio{i(x-\xt)}{\As(y-\yt)}
-2\Re\scphit{\A(x-\xt)}{y-\yt} \label{eq:Gmain2_1} .
\end{align}
The last two terms in \eqref{eq:Gmain2_1} can be written as 
(for brevity we use the notation $e:=x-\xt$ and $h:=y-\yt$)
\begin{align}
& 2\sign{\omega}\,\Re(i\scphio{e}{\As h}) -2\Re\scphit{\A e}{h} \nonumber \\
& \qquad = - 2\sign{\omega}\,\Im\scphio{e}{\As h} -2\Re\scphit{\A e}{h} \nonumber \\
& \qquad \ge - 2\norm{\Im\scphio{e}{\As h} } - 2 \norm{\Re\scphit{\A e}{h}} \nonumber \\
& \qquad = -( \ubr{\norm{\Im\scphio{e}{\As h}} +
	\norm{\Re\scphio{e}{\As h}}}_{\leq \sqrt{2}\norm{\scphio{e}{\As h}}} +
	\ubr{\norm{\Im\scphit{\A e}{h}} +
	\norm{\Re\scphit{\A e}{h}}}_{\leq \sqrt{2}\norm{\scphit{\A e}{h}}} )\nonumber \\
& \qquad \ge - \sqrt{2} \left( \norm{\scphio{e}{\As h}} + \norm{\scphit{\A e}{h}} \right) \nonumber \\
& \qquad \ge - \sqrt{2} \left( \normhioomao{e} \normhioomaomo{\As h} + 
	\normhitat{\A e} \normhitatmo{h} \right) \nonumber \\
& \qquad \ge - \sqrt{2} \left( \frac{1}{2\delta} \normhioomao{e}^2 + \frac{\delta}{2} \normhioomaomo{\As h}^2 + 
	\frac{1}{2\delta} \normhitat{\A e}^2 + \frac{\delta}{2} \normhitatmo{h}^2 \right) , \label{eq:Gmain2_2}
\end{align}
for all $\delta>0$. One can repeat these calculations by estimating from above, and arrive at
\begin{align}
& 2\sign{\omega}\,\Re(i\scphio{e}{\As h}) -2\Re\scphit{\A e}{h} \nonumber \\
& \qquad \leq \sqrt{2} \left( \frac{1}{2\delta} \normhioomao{e}^2 + \frac{\delta}{2} \normhioomaomo{\As h}^2 + 
	\frac{1}{2\delta} \normhitat{\A e}^2 + \frac{\delta}{2} \normhitatmo{h}^2 \right) . \label{eq:Gmain2_3}
\end{align}
Together \eqref{eq:Gmain2_1}--\eqref{eq:Gmain2_3} give
\begin{align}
\left(1-\sqrt{2}\frac{1}{2\delta}\right) \normdaidomaoat{x-\xt}^2 +
\left(1-\sqrt{2}\frac{\delta}{2}\right) \normdasomaomoatmo{y-\yt}^2 & \leq \Mi(\xt,\yt) , \label{eq:Gmain2_4} \\
\left(1+\sqrt{2}\frac{1}{2\delta}\right) \normdaidomaoat{x-\xt}^2 +
\left(1+\sqrt{2}\frac{\delta}{2}\right) \normdasomaomoatmo{y-\yt}^2 & \ge \Mi(\xt,\yt) . \label{eq:Gmain2_5}
\end{align}
The estimate \eqref{eq:Gmain21} follows by setting $\delta = 1$ 
in \eqref{eq:Gmain2_4} and \eqref{eq:Gmain2_5}. 
The estimate \eqref{eq:Gmain22} follows by the property \eqref{eq:solutionoperatornormgen2} 
in Remark \ref{rem:solutionoperatornormgen}, completing the proof.
\end{proof}

The square root of the ratio of the upper and lower bound in $\eqref{eq:Gmain21}$ 
is always $1+\sqrt{2} < 2.42$. The square root of the ratio 
of the bounds of the normalized counterpart $\eqref{eq:Gmain22}$ 
is always $2+\sqrt{2} < 3.42$. We can conclude that the bounds 
are close to each other and give reliable information 
of the error of a mixed approximation.

\begin{theo} \label{thm:primaldualseparate}
From the proof of Theorem \ref{thm:Gmain2} we can deduce the following 
a posteriori error estimates for the primal and dual problems.
\begin{itemize}
\item[\bf(i)]
For any $\xt\in D(\A)$ it holds
$\ds\normdaidomaoat{x-\xt}^2
\leq2\Mi(\xt,\psi)$ for any $\psi\in D(\As)$.
\item[\bf(ii)]
For any $\yt\in D(\As)$ it holds
$\ds\normdasomaomoatmo{y-\yt}^2
\leq 2 \Mi(\varphi,\yt)$ for any $\varphi\in D(\A)$.
\end{itemize}
\end{theo}

\begin{proof}
The estimate (i) follows from \eqref{eq:Gmain2_4} by setting $\delta = \sqrt{2}$, 
and (ii) from \eqref{eq:Gmain2_4} by setting $\delta = 1/\sqrt{2}$.
\end{proof}


\begin{rem}
\label{laxmilgramAsrem2}
\mbox{}
\begin{itemize}
\item[\bf(i)]
Since $y\,\bot_{\alphat^{-1}}\,N(\As)$ by \eqref{varAs2} we get immediately that $y\in\alphat\ol{R(\A)}$
by the Helmholtz decomposition $\hit=N(\As)\oplus_{\alphat^{-1}}\alphat\ol{R(\A)}$.
\item[\bf(ii)] 
If $(\omega\alphao)^{-1}f\in D(\A)$ we have $z:=(\omega\alphao)^{-1}\As y\in D(\A)$ 
and the strong and mixed formulations of \eqref{varAs2} read
\begin{align*}
\A(\omega\alphao)^{-1}\As y+i\alphat^{-1}y&=\A(\omega\alphao)^{-1}f,\\
\A z+i\alphat^{-1}y&=\A(\omega\alphao)^{-1}f,\quad (\omega\alphao)^{-1}\As y=z.
\end{align*}
Then for all $\varphi\in D(\A)$ we have
\begin{align*}
\scphitat{\A z}{\A\varphi}+i\scphioomao{z}{\varphi}
&=-i\scphit{y}{\A\varphi}+i\scphioomao{z}{\varphi}+\scphitat{\A(\omega\alphao)^{-1}f}{\A\varphi}\\
&=\scphitat{\A(\omega\alphao)^{-1}f}{\A\varphi}
\end{align*}
and hence $z\in\big(D(\A)\cap\alphao^{-1} R(\As)\big)\subset D(\A)$ 
is the unique solution of this variational problem.
Furthermore, $\alphat(\A z-\A(\omega\alphao)^{-1}f)\in D(\As)$ 
and $\As\alphat(\A z-\A\alphao^{-1}f)=-i\omega\alphao z$.
If $\alphat\A(\omega\alphao)^{-1}f$ belongs to $D(\As)$, this yields 
$\alphat\A z\in D(\As)$ and the strong equation
$$\As\alphat\A z+i\omega\alphao z=\As\alphat\A(\omega\alphao)^{-1}f.$$
\end{itemize}
\end{rem}


\subsection{Error Indication Properties for PDEs}
\label{subsec:ind}

In this section we assume that the underlying problem is a PDE 
such that $\A$ and $\As$ are differential operators and the Hilbert spaces are scalar, 
vector, or tensor valued $\lt$-spaces, i.e., $\hio=\lt(\om)$ and $\hit=\lt(\om)$. 
Here $\om\subset\reals^d$, $d\geq1$, is a domain.

Let $\elems$ denote a discretization of the domain $\om$ 
into a mesh of non-overlapping elements $\elem$. 
Note that we assume $\bigcup_{\elem\in\elems} \ol{\elem} = \ol{\Omega}$, 
i.e., in particular that the boundary of $\Omega$ is exactly represented by the mesh. 
This is necessary in order to have conforming approximations in the first place: 
They must satisfy exactly the imposed boundary conditions.

Aside from golbal error values we are also interested in estimating 
the error distribution in the mesh $\elems$. In the following we use 
the previously derived error equality and error estimate to define error indicators 
and study their properties.

\subsubsection*{Case I}

We define the following error indicator based on the equality of Theorem \ref{thm:Gmain}:
$$\eta_\elem(\xt,\yt) :=
\sqrt{\norm{f-\alphao\xt-\As\yt}_{\lt(\elem),\alphao^{-1}}^2
+\norm{\yt-\alphat\A\xt}_{\lt(\elem),\alphat^{-1}}^2} $$
The error indicator $\eta_\elem$ will indicate the exact error distribution
$$e_\elem(\xt,\yt)
:=\sqrt{ \norm{x-\xt}_{\lt(\elem),\alphao}^2
+\norm{\A(x-\xt)}_{\lt(\elem),\alphat}^2
+\norm{y-\yt}_{\lt(\elem),\alphat^{-1}}^2
+\norm{\As(y-\yt)}_{\lt(\elem),\alphao^{-1}}^2 }.$$
In the following
$$\eta:=\sqrt{\sum_{\elem\in\elems} \eta_\elem^2}
\qquad \textrm{and} \qquad
e:=\sqrt{\sum_{\elem\in\elems} e_\elem^2} .$$
The error indicator $\eta$ should satisfy the following properties:
\begin{enumerate}
\item The indicator $\eta$ must satisfy the global relation
$\ul{c}\,\eta \leq e \leq \ol{c}\,\eta$
with some constants $\ul{c}>0$ and $\ol{c}>0$. 
The constant $\ul{c}$ is often called the global efficiency constant, 
and $\ol{c}$ the global reliability constant. 
If $\ol{c}$, or an upper bound of it is known, 
the indicator can be used to provide a stopping criterion for adaptive computations.
\item The local indicator $\eta_\elem$ must satisfy
$c_\elem \eta_\elem \leq e_{\elem}$
in all elements $\elem$ in $\elems$ with some constants $c_\elem>0$, 
which are often called the local efficiency constants. 
If $c_\elem$ are of the same magnitude, 
the indicator is then appropriate for estimating 
the error distribution in the mesh, and can then be used for adaptive mesh-refinement.
\end{enumerate}
It is desirable that the constants $\ul{c},\ol{c}$ and $c_\elem$ 
are not dependent on the problem data or the mesh. 
If the constants $\ul{c}$ and $\ol{c}$ are known, 
they give a good idea of the quality of the indicator $\eta$ in a global context. 
It is also desirable that the local constants $c_\elem$ are known for all elements $\elem$. 
The closer the values are to $\ul{c}$, the better.

Note that $\eta(\xt,\yt) = \M(\xt,\yt)^{\nicefrac{1}{2}} = e(\xt,\yt)$, 
so according to Theorem \ref{thm:Gmain} the first property 
is satisfied with constants $\ul{c}=\ol{c}=1$.
This is the best case possible.

We show the second property of local efficiency
by using \eqref{eq:Gstrong} and inserting $0=\alphat\A x-y$ into $\eta_\elem$:
\begin{align*}
\eta_\elem(\xt,\yt)^2
&=\norm{\alphao x-\alphao\xt+\As y-\As\yt}_{\lt(\elem),\alphao^{-1}}^2
+\norm{\yt-y+\alphat\A x-\alphat\A\xt}_{\lt(\elem),\alphat^{-1}}^2\\
&\leq 2 \left( \norm{x-\xt}_{\lt(\elem),\alphao}^2
+\norm{\As(y-\yt)}_{\lt(\elem),\alphao^{-1}}^2
+\norm{\yt-y}_{\lt(\elem),\alphat^{-1}}^2
+\norm{\A(x-\xt)}_{\lt(\elem),\alphat}^2 \right),
\end{align*}
which gives us
$$\frac{1}{\sqrt{2}} \, \eta_\elem(\xt,\yt) \leq e_\elem(\xt,\yt).$$
The indicator $\eta$ then satisfies the second property 
with the constant $c_\elem = 1/\sqrt{2} > 0.7$ for all elements $\elem\in\elems$. 
This constant is rather sharp, since $\ul{c} = \ol{c} = 1$. 
This means that $\eta$ provides a good error indicator 
for guiding mesh-adaptive methods for mixed approximations.

\subsubsection*{Case II}

We define the following error indicator based on the estimate of Theorem \ref{thm:Gmain2}:
$$\eta_{i,\elem}(\xt,\yt) :=
\sqrt{\norm{f-i\omega\alphao\xt-\As\yt}_{\lt(\elem),(\norm{\omega}\alphao)^{-1}}^2
+\norm{\yt-\alphat\A\xt}_{\lt(\elem),\alphat^{-1}}^2}$$
The error indicator $\eta_{i,\elem}$ will indicate the exact error distribution
$$e_{i,\elem}(\xt,\yt)
:= \sqrt{ \norm{x-\xt}_{\lt(\elem),\norm{\omega}\alphao}^2
+\norm{\A(x-\xt)}_{\lt(\elem),\alphat}^2 \\
+\norm{y-\yt}_{\lt(\elem),\alphat^{-1}}^2
+\norm{\As(y-\yt)}_{\lt(\elem),(\norm{\omega}\alphao)^{-1}}^2 }.$$
In the following
$$\eta_i:=\sqrt{\sum_{\elem\in\elems} \eta_{i,\elem}^2}
\qquad \textrm{and} \qquad
e_i:=\sqrt{\sum_{\elem\in\elems} e_{i,\elem}^2} .$$
Note that $\eta_i(\xt,\yt) = \Mi(\xt,\yt)^{\nicefrac{1}{2}}$, 
so according to Theorem \ref{thm:Gmain2} the first property is satisfied with constants 
$$	\ul{c}=\sqrt{\frac{\sqrt{2}}{\sqrt{2}+1}} > 0.76
	\qquad \textrm{and} \qquad
	\ol{c}=\sqrt{\frac{\sqrt{2}}{\sqrt{2}-1}} < 1.85,$$
with ratio $1+\sqrt{2} < 2.42$.

We show the second property of local efficiency
by using \eqref{eq:Gstrong2} and inserting $0=\alphat\A x-y$ into $\eta_{i,\elem}$:
\begin{align*}
\eta_{i,\elem}(\xt,\yt)^2
&=\norm{i\omega\alphao x-i\omega\alphao\xt+\As y-\As\yt}_{\lt(\elem),(\norm{\omega}\alphao)^{-1}}^2
+\norm{\yt-y+\alphat\A x-\alphat\A\xt}_{\lt(\elem),\alphat^{-1}}^2\\
&\leq 2 \left( \norm{x-\xt}_{\lt(\elem),\norm{\omega}\alphao}^2
+\norm{\As(y-\yt)}_{\lt(\elem),(\norm{\omega}\alphao)^{-1}}^2
+\norm{\yt-y}_{\lt(\elem),\alphat^{-1}}^2
+\norm{\A(x-\xt)}_{\lt(\elem),\alphat}^2 \right),
\end{align*}
which gives us
$$\frac{1}{\sqrt{2}} \, \eta_{i,\elem}(\xt,\yt) \leq e_{i,\elem}(\xt,\yt) .$$
The indicator $\eta_i$ then satisfies the second property 
with the constant $c_\elem = 1/\sqrt{2} > 0.7$ for all elements $\elem\in\elems$. 
This constant is again rather sharp, since $0.76 < \ul{c} < \ol{c} < 1.85$. 
This means that $\eta_i$ provides a good error indicator 
for guiding mesh-adaptive methods for mixed approximations.


\subsection{Motivation: Error Control for Time Dependent PDEs}
\label{subsec:timediscr}

As mentioned in the introduction, a motivation to study a posteriori error estimation 
for the two classes of problems considered in this paper 
comes from time dependent partial differential equations, more precisely
from their time discretizations.

\subsubsection*{Case I}

A main application of our error equality of Theorem \ref{thm:Gmain} might be that equations of the type
\begin{align}
\label{AsAeqapp}
\As\alphat\A x+\alphao x&=f
\end{align}
naturally occur in many types of time discretizations, 
e.g. for linear parabolic heat type equations or linear hyperbolic wave propagation type equations.

Let us consider the linear parabolic heat type equation
\begin{align}
\label{dtAsA}
(\p_{t}+\As\A)x=f,
\end{align}
generalizing the most prominent example of the heat equation
$$(\p_{t}-\Delta)u
=(\p_{t}-\div\na)u=g$$
with appropriate boundary and initial conditions.
A standard implizit time discretization for \eqref{dtAsA} is e.g. the backward Euler scheme, yielding
$$\delta_{n}^{-1}(x_{n}-x_{n-1})+\As\A x_{n}=f_{n},\quad
\delta_{n}:=t_{n}-t_{n-1},$$
and hence \eqref{AsAeqapp} is recovered by
$$\As\A x_{n}+\delta_{n}^{-1}x_{n}=\tilde{f}_{n}:=f_{n}-\delta_{n}^{-1}x_{n-1}.$$
We note that our arguments extend to `all' practically used time discretizations.
Functional a posteriori error estimates for parabolic equations can be found e.g.
in \cite{repinbookone,NeittaanmakiRepin2004}.

A large class of linear wave propagation models, like electromagnetics or acoustics, have the structure 
\begin{align}
\label{dtM}
(\p_{t}\Lambda^{-1}+\Max)\begin{bmatrix}x\\y\end{bmatrix}
=\begin{bmatrix}g\\h\end{bmatrix},\quad
\Max=\begin{bmatrix}0&-\As\\\A&0\end{bmatrix},\quad
\Lambda=\begin{bmatrix}\lambda_{1}&0\\0&\lambda_{2}\end{bmatrix}
\end{align}
or more explicit
\begin{align}
\label{dteq}
\p_{t}\lambda_{1}^{-1}x-\As y=g,\quad
\p_{t}\lambda_{2}^{-1}y+\A x=h
\end{align}
completed by appropriate initial conditions.
Often the material is assumed to be time-independent, i.e., $\Lambda$ does not depend on time.
In this case $i\Lambda\Max$ is selfadjoint in the proper Hilbert spaces 
and the solution theory follows immediately 
by the spectral theorem (variation in constant formula) or by semi group theory.
We note that formally the second order wave equation 
$$\big(\p_{t}^2-(\Lambda\Max)^2\big)\begin{bmatrix}x\\y\end{bmatrix}
=\begin{bmatrix}\tilde{g}\\\tilde{h}\end{bmatrix}
:=(\p_{t}-\Lambda\Max)\Lambda\begin{bmatrix}g\\h\end{bmatrix},\quad
(\Lambda\Max)^2
=\begin{bmatrix}-\lambda_{1}\As\lambda_{2}\A&0\\0&-\lambda_{2}\A\lambda_{1}\As\end{bmatrix}$$
holds, i.e., component wise 
$$(\p_{t}^2+\lambda_{1}\As\lambda_{2}\A)x=\tilde{g},\quad
(\p_{t}^2+\lambda_{2}\A\lambda_{1}\As)y=\tilde{h}.$$
Hence the linear hyperbolic wave type equation
\begin{align}
\label{dtsqAsA}
(\p_{t}^2+\As\A)x=f
\end{align}
pops up, generalizing the most prominent example of the wave equation
$$(\p_{t}^2-\Delta)u
=(\p_{t}^2-\div\na)u=j$$
with appropriate boundary and initial conditions.
A standard implizit time discretization for \eqref{dteq} is e.g. the backward Euler scheme, i.e.,
$$\delta_{n}^{-1}\lambda_{1}^{-1}(x_{n}-x_{n-1})-\As y_{n}=g_{n},\quad
\delta_{n}^{-1}(y_{n}-y_{n-1})+\lambda_{2}\A x_{n}=\lambda_{2}h_{n}.$$
Hence, we obtain e.g. for $x_{n}$
$$\As\lambda_{2}\A x_{n}+\delta_{n}^{-2}\lambda_{1}^{-1}x_{n}
=f_{n}:=\As(\lambda_{2}h_{n}+\delta_{n}^{-1}y_{n-1})+\delta_{n}^{-2}\lambda_{1}^{-1}x_{n-1} + \delta_n^{-1} g_n$$
provided that $\lambda_{2}h_{n}\in D(\As)$.
Therefore \eqref{AsAeqapp} holds for $x_{n}$ with e.g. 
$\alphao=\delta_{n}^{-2}\lambda_{1}^{-1}$ and $\alphat=\lambda_{2}$.
Of course, a similar equation holds for $y_{n}$ as well.
We note that our arguments extend to `all' practically used time discretizations.
Functional a posteriori error estimates for wave equations can be found 
in \cite{repinapostwave,paulyrepinrossihypmax}.

\subsubsection*{Case II}

A main application of our two-sided error estimate of Theorem \ref{thm:Gmain2} 
might be that equations of the type
\begin{align}
\label{AsAieqapp}
\As\alphat\A x+i\omega\alphao x&=f
\end{align}
naturally occur in time discretizations for the eddy current problem in electromagnetics.
Maxwell's equations are hyperbolic and read
\begin{align*}
\p_{t}D-\rot H&=J=j+\sigma E,&
\div D&=\rho,&
D&=\eps E,\\
\p_{t}B+\rot E&=0,&
\div B&=0,&
B&=\mu H
\end{align*}
with appropriate boundary and initial conditions. These equations
can be written in the style of \eqref{dtM} as
\begin{align*}
\Big(\p_{t}\begin{bmatrix}\eps&0\\0&\mu\end{bmatrix}
+\begin{bmatrix}0&-\rot\\\rot&0\end{bmatrix}
-\begin{bmatrix}\sigma&0\\0&0\end{bmatrix}\Big)
\begin{bmatrix}E\\H\end{bmatrix}
&=\begin{bmatrix}j\\0\end{bmatrix}.
\end{align*}
Let us assume that $\eps$, $\mu$ and $\sigma$ are independent of time. Then, formally, we have
\begin{align*}
\p_{t}^2\eps E
&=\rot\mu^{-1}\p_{t}B+\p_{t}j+\p_{t}\sigma E
=-\rot\mu^{-1}\rot E+\p_{t}J,\\
\p_{t}^2\mu H
&=-\rot\eps^{-1}\p_{t}D
=-\rot\eps^{-1}\rot H-\rot\eps^{-1}J,
\end{align*}
i.e., we get the wave equations
$$(\p_{t}^2
+\eps^{-1}\rot\mu^{-1}\rot)E
=\p_{t}\eps^{-1}J,\quad
(\p_{t}^2
+\mu^{-1}\rot\eps^{-1}\rot)H
=-\mu^{-1}\rot\eps^{-1}J$$
as another example of \eqref{dtsqAsA}.
The eddy current model neglects time variations of the electric field, i.e., assumes $\p_{t}D=\p_{t}\eps E=0$,
and hence leads to the parabolic equation
$$\sigma\p_{t}E
=-\rot\mu^{-1}\p_{t}B-F
=\rot\mu^{-1}\rot E-F,\quad
F:=\p_{t}j,$$
i.e.,
$$-\sigma\p_{t}E+\rot\mu^{-1}\rot E=F.$$
A time-harmonic ansatz leads to 
$$\rot\mu^{-1}\rot\tilde{E}
+i\omega\sigma\tilde{E}
=\tilde{F}$$
as a prominent example of \eqref{AsAieqapp}.


\section{Applications} \label{sec:A}

In this section we will discuss some standard applications.
Let $\om\subset\reals^d$, $d\geq1$, be a bounded Lipschitz domain with boundary $\ga$.
Moreover, let $\gad$ be an open subset of $\ga$ and $\gan:=\ga\setminus\ol{\gad}$
its complement. We will denote by $n$ the outward unit normal of the boundary $\ga$.
We note that our results extend to unbounded domains without any changes.

We denote by $\scplt{\,\cdot\,}{\,\cdot\,}$ and $\normlt{\,\cdot\,}$ the inner product and the norm 
in $\lt$ for scalar-, vector- and matrix-valued functions. 
Throughout this section we will not indicate the dependence on $\om$
in our notations of the functional spaces.

For the first application, the reaction-diffusion problem, 
we repeat all the results of Section \ref{sec:G}. 
For the rest of the applications we will repeat 
only the main results of Theorems \ref{thm:Gmain} 
and \ref{thm:Gmain2} for the sake of brevity.


\subsection{Reaction-Diffusion} \label{subsec:RD}

We define the usual Sobolev spaces\footnote{The space $\d$ is often denoted by $\hilbert(\div)$ in the literature.}
$$\ho:=\set{\varphi\in\lt}{\na\varphi\in\lt},\quad\d:=\set{\psi\in\lt}{\div\psi\in\lt},$$
and the spaces
$$\hogad:=\ol{\cigad}^{\ho},\quad
\dgan:=\ol{\cigan}^{\d},$$
were $\cigad$ resp. $\cigan$ is the space of smooth test functions resp. vector fields 
having supports bounded away from $\gad$ resp. $\gan$.
These are Hilbert spaces equipped with the graph norms denoted by
$\normho{\,\cdot\,}$, $\normd{\,\cdot\,}$, respectively.
The following table shows the relation to the notation of Section \ref{sec:G}:
\begin{center}\begin{tabular}{c|c||c|c||c|c||c|c}
$\alphao$ & $\alphat$ & $\A$ & $\As$ & $\hio$ & $\hit$ & $D(\A)$ & $D(\As)$ \\
\hline
$\rho$ & $\Amat$ & $\na$ & $-\div$ & $\lt$ & $\lt$ & $\hogad$ & $\dgan$
\end{tabular}\end{center}
We note that indeed $D(\As)=\dgan$ holds for Lipschitz domains, see e.g. 
\cite{jochmanncompembmaxmixbc,bauerpaulyschomburgmcpweaklip}. 
The relation \eqref{partint} reads now
$$\forall\,\varphi\in\hogad\quad\forall\,\psi\in\dgan\qquad
\scplt{\na\varphi}{\psi}=-\scplt{\varphi}{\div\psi}.$$

\subsubsection*{Case I}

Find the scalar potential $u\in\ho$, such that
\begin{align}
-\div\Amat\na u+\rho\,u&=f&
\textrm{in }&\om,\nonumber\\
u&=0&
\textrm{on }&\gad,\label{eq:RD}\\
n\cdot\Amat\na u&=0&
\textrm{on }&\gan\nonumber.
\end{align}
The quadratic diffusion matrix $\Amat\in\li$ is symmetric, real valued, and uniformly positive definite. 
The real valued reaction coefficient $\rho\ge\rho_0>0$ belongs to $\li$
and the source $f$ to $\lt$. The dual variable for this problem is the flux $p=\Amat\na u\in\d$.
The mixed formulation of \eqref{eq:RD} reads: Find $(u,p)\in\hogad\times\dgan$ such that
\begin{align}
\label{eq:RDMixed}
-\div p+\rho\,u=f,\quad\Amat\na u=p \qquad \textrm{in } \om.
\end{align}
The primal and dual variational problems are: Find $(u,p)\in\hogad\times\dgan$ such that
\begin{align*}
\forall\,\varphi&\in\hogad&
\scp{\na u}{\na\varphi}_{\lt,\Amat}+\scp{u}{\varphi}_{\lt,\rho}
&=\scplt{f}{\varphi},\\
\forall\,\psi&\in\dgan&
\scp{\div p}{\div\psi}_{\lt,\rho^{-1}}+\scp{p}{\psi}_{\lt,\Amat^{-1}}
&=-\scp{f}{\div\psi}_{\lt,\rho^{-1}}.
\end{align*}
Considering the norms we have
\begin{align*}
\norm{u}_{\ho,\rho,\Amat}^2
&=\norm{u}_{\lt,\rho}^2
+\norm{\na u}_{\lt,\Amat}^2,\\
\norm{p}_{\d,\rho^{-1},\Amat^{-1}}^2
&=\norm{p}_{\lt,\Amat^{-1}}^2
+\norm{\div p}_{\lt,\rho^{-1}}^2,\\
\tnorm{(u,p)}^2
&=\norm{u}_{\ho,\rho,\Amat}^2
+\norm{p}_{\d,\rho^{-1},\Amat^{-1}}^2.
\end{align*}

Now Remark \ref{rem:isometrygen}, Theorem \ref{thm:Gmain_reg}, 
Theorem \ref{thm:Gmain}, and Corollary \ref{cor:maintheocor} read:

\begin{rem}
We note $\norm{u}_{\ho,\rho,\Amat}\leq\norm{f}_{\lt,\rho^{-1}}$ 
and $\norm{p}_{\d,\rho^{-1},\Amat^{-1}}\leq\norm{f}_{\lt,\rho^{-1}}$ and indeed
$$\tnorm{(u,p)}=\norm{f}_{\lt,\rho^{-1}}.$$
The solution operator $L:\lt\to\hogad\times\dgan;f\mapsto(u,p)$ is an isometry, i.e. $\norm{L}=1$.
\end{rem}

\begin{theo}
\label{thm:RD_reg}
Let $(u,p)\in\hogad\times\dgan$ be the exact solution of \eqref{eq:RDMixed}.
Let $\ut \in \hogad$ and $\pt=\alpha\na\ut \in \dgan$.
Then, for the mixed approximation $(\ut,\pt)$ we have
$$\tnorm{(u,p)-(\ut,\pt)}^2
=\Rrd(\ut,\pt),\quad
\frac{\tnorm{(u,p)-(\ut,\pt)}^2}{\tnorm{(u,p)}^2}
=\frac{\Rrd(\ut,\pt)}{\norm{f}_{\lt,\rho^{-1}}^2},$$
where $\Rrd(\ut,\pt)=\norm{f-\rho\,\ut+\div\pt}_{\lt,\rho^{-1}}^2$.
\end{theo}

\begin{theo} 
\label{thm:RD}
Let $(u,p),(\ut,\pt)\in\hogad\times\dgan$ 
be the exact solution of \eqref{eq:RDMixed} 
and any approximation, respectively. Then
$$\tnorm{(u,p)-(\ut,\pt)}^2
=\Mrd(\ut,\pt),\quad
\frac{\tnorm{(u,p)-(\ut,\pt)}^2}{\tnorm{(u,p)}^2}
=\frac{\Mrd(\ut,\pt)}{\norm{f}_{\lt,\rho^{-1}}^2}$$
hold, where $\Mrd(\ut,\pt)=\norm{f-\rho\,\ut+\div\pt}_{\lt,\rho^{-1}}^2
+\norm{\pt-\Amat\na\ut}_{\lt,\Amat^{-1}}^2$.
\end{theo}

\begin{cor}
Theorem \ref{thm:RD} provides the well known a posteriori error estimates 
for the primal and dual problems. 
\begin{itemize}
\item[\bf(i)]
For any $\ut\in\hogad$ it holds
$\ds\norm{u-\ut}_{\ho,\rho,\Amat}^2
=\min_{\psi\in\dgan}\Mrd(\ut,\psi)
=\Mrd(\ut,p)$.
\item[\bf(ii)]
For any $\pt\in\dgan$ it holds
$\ds\norm{p-\pt}_{\d,\rho^{-1},\Amat^{-1}}^2
=\min_{\varphi\in\hogad}\Mrd(\varphi,\pt)
=\Mrd(u,\pt)$.
\end{itemize}
\end{cor}

Error indication properties of Section \ref{subsec:ind} hold as well:

\begin{rem}
Let $\elems$ denote a discretization of the domain $\om$ 
into a mesh of non-overlapping elements $\elem$ 
such as described in Section \ref{subsec:ind}. 
We define the following error indicator using 
the functional of Theorem \ref{thm:RD}:
$$\eta_\elem(\ut,\pt) :=
\sqrt{\norm{f-\rho\,\ut+\div\pt}_{\lt(\elem),\rho^{-1}}^2
+\norm{\pt-\Amat\na\ut}_{\lt(\elem),\Amat^{-1}}^2},
\qquad
\eta:=\sqrt{\sum_{\elem\in\elems} \eta_\elem^2}$$
The error indicator $\eta$ will indicate the exact error distribution
$$e_\elem(\xt,\yt)
:=\sqrt{ \norm{u-\ut}_{\ho(T),\rho,\Amat}^2
+\norm{p-\pt}_{\d(T),\rho^{-1},\Amat^{-1}}^2 },
\qquad
e:=\sqrt{\sum_{\elem\in\elems} e_\elem^2}.$$
As shown in Section \ref{subsec:ind}, the global reliability constant, 
global efficiency constant, and the local efficiency constants are
$$	\ol{c} = 1, \qquad \ul{c} = 1, \qquad c_\elem = \frac{1}{\sqrt{2}} > 0.7 \quad \forall\,\elem\in\elems,$$
respectively.
\end{rem}

Related results and numerical tests for exterior domains can be found in e.g.
\cite{paulyrepinell,malimuzalevskiypaulyextell}.

\subsubsection*{Case II}

Find the scalar potential $u\in\ho$, such that
\begin{align}
-\div\Amat\na u+i\omega\rho\,u&=f&
\textrm{in }&\om,\nonumber\\
u&=0&
\textrm{on }&\gad,\label{eq:RDi}\\
n\cdot\Amat\na u&=0&
\textrm{on }&\gan\nonumber,
\end{align}
where $\Amat,\rho$, and $f$ are as before, and $\omega \in \realsnotz$.
The dual variable for this problem is the flux $p=\Amat\na u\in\d$.
The mixed formulation of \eqref{eq:RDi} reads: Find $(u,p)\in\hogad\times\dgan$ such that
\begin{equation}
\label{eq:RDiMixed}
-\div p+i\omega\rho\,u=f,\quad\Amat\na u=p \qquad \textrm{in } \om.
\end{equation}
Considering the norms we have
\begin{align*}
\norm{u}_{\ho,\norm{\omega}\rho,\Amat}^2
&=\norm{u}_{\lt,\norm{\omega}\rho}^2
+\norm{\na u}_{\lt,\Amat}^2,\\
\norm{p}_{\d,(\norm{\omega}\rho)^{-1},\Amat^{-1}}^2
&=\norm{p}_{\lt,\Amat^{-1}}^2
+\norm{\div p}_{\lt,(\norm{\omega}\rho)^{-1}}^2,\\
\ttnorm{(u,p)}^2
&=\norm{u}_{\ho,\norm{\omega}\rho,\Amat}^2
+\norm{p}_{\d,(\norm{\omega}\rho)^{-1},\Amat^{-1}}^2.
\end{align*}
The primal and dual variational problems are: Find $(u,p)\in\hogad\times\dgan$ such that
\begin{align*}
\forall\,\varphi&\in\hogad&
\scp{\na u}{\na\varphi}_{\lt,\Amat}+i\scp{u}{\varphi}_{\lt,\omega\rho}
&=\scplt{f}{\varphi},\\
\forall\,\psi&\in\dgan&
\scp{\div p}{\div\psi}_{\lt,(\omega\rho)^{-1}}+i\scp{p}{\psi}_{\lt,\Amat^{-1}}
&=-\scp{f}{\div\psi}_{\lt,(\omega\rho)^{-1}}.
\end{align*}

Now Remark \ref{rem:solutionoperatornormgen}, Theorem \ref{thm:Gmain2_reg}, 
Theorem \ref{thm:Gmain2}, and Theorem \ref{thm:primaldualseparate} read:

\begin{rem}
We note $\norm{u}_{\ho,\norm{\omega}\rho,\Amat}\leq\sqrt{2}\norm{f}_{\lt,(\norm{\omega}\rho)^{-1}}$ 
and $\norm{p}_{\d,(\norm{\omega}\rho)^{-1},\Amat^{-1}}\leq\sqrt{2}\norm{f}_{\lt,(\norm{\omega}\rho)^{-1}}$ 
and indeed
$$\norm{f}_{\lt,(\norm{\omega}\rho)^{-1}}\leq\ttnorm{(u,p)}\leq\sqrt{2}\norm{f}_{\lt,(\norm{\omega}\rho)^{-1}}.$$
The norm of the solution operator $L_i:\lt\to\hogad\times\dgan;f\mapsto(u,p)$ 
then satisfies $1\leq\norm{L_i}\leq\sqrt{2}$.
\end{rem}

\begin{theo}
\label{thm:RDi_reg}
Let $(u,p)\in\hogad\times\dgan$ be the exact solution of \eqref{eq:RDMixed}.
Let $\ut \in \hogad$ and $\pt=\alpha\na\ut \in \dgan$.
Then, for the mixed approximation $(\ut,\pt)$ we have
$$\Rrdi(\ut,\pt)
\leq \ttnorm{(u,p)-(\ut,\pt)}^2
\leq 2 \Rrdi(\ut,\pt)$$
and
$$\frac{1}{2} \cdot \frac{\Rrdi(\ut,\pt)}{\norm{f}_{\lt,(\norm{\omega}\rho)^{-1}}^2}	
\leq \frac{\ttnorm{(u,p)-(\ut,\pt)}^2}{\ttnorm{(u,p)}^2}
\leq 2 \frac{\Rrdi(\ut,\pt)}{\norm{f}_{\lt,(\norm{\omega}\rho)^{-1}}^2} ,	$$
where $\Rrdi(\ut,\pt)=\norm{f-i\omega\rho\,\ut+\div\pt}_{\lt,(\norm{\omega}\rho)^{-1}}^2$.
\end{theo}

\begin{theo} 
\label{thm:RDi}
Let $(u,p),(\ut,\pt)\in\hogad\times\dgan$ 
be the exact solution of \eqref{eq:RDiMixed} 
and any approximation, respectively. Then
$$\frac{\sqrt{2}}{\sqrt{2}+1} \Mrdi(\ut,\pt)
\leq \ttnorm{(u,p)-(\ut,\pt)}^2
\leq \frac{\sqrt{2}}{\sqrt{2}-1} \Mrdi(\ut,\pt)$$
and
$$\frac{\sqrt{2}}{2(\sqrt{2}+1)} \cdot \frac{\Mrdi(\ut,\pt)}{\norm{f}_{\lt,(\norm{\omega}\rho)^{-1}}^2}	
\leq \frac{\ttnorm{(u,p)-(\ut,\pt)}^2}{\ttnorm{(u,p)}^2}
\leq \frac{\sqrt{2}}{\sqrt{2}-1} \cdot \frac{\Mrdi(\ut,\pt)}{\norm{f}_{\lt,(\norm{\omega}\rho)^{-1}}^2}	$$
hold, where $\Mrdi(\ut,\pt)=\norm{f-i\omega\rho\,\ut+\div\pt}_{\lt,(\norm{\omega}\rho)^{-1}}^2
+\norm{\pt-\Amat\na\ut}_{\lt,\Amat^{-1}}^2$.
\end{theo}

\begin{theo}
We have the following a posteriori error estimates for the primal and dual problems.
\begin{itemize}
\item[\bf(i)]
For any $\ut\in\hogad$ it holds
$\ds\norm{u-\ut}_{\ho,\norm{\omega}\rho,\Amat}^2
\leq2\Mrdi(\ut,\psi)$ for any $\psi \in \dgan$.
\item[\bf(ii)]
For any $\pt\in\dgan$ it holds
$\ds\norm{p-\pt}_{\d,(\norm{\omega}\rho)^{-1},\Amat^{-1}}^2
\leq2\Mrdi(\varphi,\pt)$ for any $\varphi \in \hogad$.
\end{itemize}
\end{theo}

Error indication properties of Section \ref{subsec:ind} hold as well:

\begin{rem}
Let $\elems$ denote a discretization of the domain $\om$ 
into a mesh of non-overlapping elements $\elem$ such as described in Section \ref{subsec:ind}. 
We define the following error indicator using the functional of Theorem \ref{thm:RDi}:
$$\eta_{i,\elem}(\ut,\pt) :=
\sqrt{\norm{f-i\omega\rho\,\ut+\div\pt}_{\lt(\elem),(\norm{\omega}\rho)^{-1}}^2
+\norm{\pt-\Amat\na\ut}_{\lt(\elem),\Amat^{-1}}^2},
\qquad
\eta_i:=\sqrt{\sum_{\elem\in\elems} \eta_{i,\elem}^2}$$
The error indicator $\eta_i$ will indicate the exact error distribution
$$e_{i,\elem}(\xt,\yt)
:= \sqrt{ \norm{u-\ut}_{\ho(T),\norm{\omega}\rho,\Amat}^2
+\norm{p-\pt}_{\d(T),(\norm{\omega}\rho)^{-1},\Amat^{-1}}^2 },
\qquad
e_i:=\sqrt{\sum_{\elem\in\elems} e_{i,\elem}^2}.$$
As shown in Section \ref{subsec:ind}, the global reliability constant, 
global efficiency constant, and the local efficiency constants are
$$\ol{c} = \sqrt{\frac{\sqrt{2}}{\sqrt{2}-1}} < 1.85 , \qquad
\ul{c} = \sqrt{\frac{\sqrt{2}}{\sqrt{2}+1}} > 0.76 , \qquad
c_\elem = \frac{1}{\sqrt{2}} > 0.7 \quad \forall\,\elem\in\elems,$$
respectively.
\end{rem}


\subsection{Electro-Magnetic Problems (3D)} \label{subsec:EC}

Let $d=3$. We need the Sobolev spaces\footnote{The space $\r$ 
is often denoted by $\hilbert(\rot)$ or $\hilbert(\textrm{curl})$ in the literature.}
$$\r:=\set{\Phi\in\lt}{\rot\Phi\in\lt},\quad
\rgad:=\ol{\cigad}^{\r},\quad
\rgan:=\ol{\cigan}^{\r} .$$
The following table shows the relation to the notation of Section \ref{sec:G}:
\begin{center}\begin{tabular}{c|c||c|c||c|c||c|c}
$\alphao$ & $\alphat$ & $\A$ & $\As$ & $\hio$ & $\hit$ & $D(\A)$ & $D(\As)$\\
\hline
$\eps,\sigma$ & $\mu^{-1}$ & $\rot$ & $\rot$ & $\lt$ & $\lt$ & $\rgad$ & $\rgan$
\end{tabular}\end{center}
We note that indeed $D(\As)=\rgan$ holds for Lipschitz domains, see e.g. 
\cite{jochmanncompembmaxmixbc,bauerpaulyschomburgmcpweaklip}. 
The relation \eqref{partint} reads now
$$\forall\,\Phi\in\rgad\quad\forall\,\Psi\in\rgan\qquad
\scplt{\rot\Phi}{\Psi}=\scplt{\Phi}{\rot\Psi}.$$

\subsubsection*{Case I: a Maxwell Type Problem}

The problem reads: Find the electric field $E\in\r$ such that
\begin{align}
\rot\mu^{-1}\rot E+\eps E&=J&
\textrm{in }&\om,\nonumber\\
n\times E&=0&
\textrm{on }&\gad,\label{eq:EC}\\
n\times\mu^{-1}\rot E&=0&
\textrm{on }&\gan.\nonumber
\end{align}
We assume that the magnetic permeability $\mu$ and the electric permittivity $\eps$ 
are symmetric, real valued, and uniformly positive definite matrices from $\li$.
The electric current $J$ belongs to $\lt$. 
The dual variable for this problem is the magnetic field $H=\mu^{-1}\rot E\in\r$.
The mixed formulation of \eqref{eq:EC} reads as follows: Find $(E,H)\in\rgad\times\rgan$ such that
\begin{equation} \label{eq:ECMixed}
\rot H+\eps E=J,\quad\mu^{-1}\rot E=H \qquad \textrm{in } \om.	
\end{equation}
Considering the norms we have
\begin{align*}
\norm{E}_{\r,\eps,\mu^{-1}}^2
&=\norm{E}_{\lt,\eps}^2
+\norm{\rot E}_{\lt,\mu^{-1}}^2,\\
\norm{H}_{\r,\eps^{-1},\mu}^2
&=\norm{H}_{\lt,\mu}^2
+\norm{\rot H}_{\lt,\eps^{-1}}^2,\\
\tnorm{(E,H)}^2
&=\norm{E}_{\r,\eps,\mu^{-1}}^2
+\norm{H}_{\r,\eps^{-1},\mu}^2.
\end{align*}
Now Theorem \ref{thm:Gmain} reads:
\begin{theo} 
\label{thm:EC}
Let $(E,H),(\Et,\Ht)\in\rgad\times\rgan$ 
be the exact solution of \eqref{eq:ECMixed} 
and any approximation, respectively. Then
$$\tnorm{(E,H)-(\Et,\Ht)}^2
=\Mec(\Et,\Ht),\quad
\frac{\tnorm{(E,H)-(\Et,\Ht)}^2}{\tnorm{(E,H)}^2}
=\frac{\Mec(\Et,\Ht)}{\norm{J}_{\lt,\eps^{-1}}^2}$$
hold, where $\Mec(\Et,\Ht)=\norm{J-\eps\Et-\rot\Ht}_{\lt,\eps^{-1}}^2
+\norm{\Ht-\mu^{-1}\rot\Et}_{\lt,\mu}^2$.
\end{theo}

Earlier results for eddy current and static Maxwell problems can be found in 
\cite{anjammalimuzalevskiyneittaanmakirepinmaxtype,paulyrepinmaxst}.

\subsubsection*{Case II: Eddy-Current}

The problem reads: Find the electric field $E\in\r$ such that
\begin{align}
\rot\mu^{-1}\rot E+i\omega\sigma E&=J&
\textrm{in }&\om,\nonumber\\
n\times E&=0&
\textrm{on }&\gad,\label{eq:ECi}\\
n\times\mu^{-1}\rot E&=0&
\textrm{on }&\gan,\nonumber
\end{align}
where $\mu$ and $J$ are as before, the conductivity $\sigma$ 
is a symmetric, real valued, and uniformly positive definite matrix from $\li$,
and $\omega \in \realsnotz$.
The dual variable for this problem is the magnetic field $H=\mu^{-1}\rot E\in\r$.
The mixed formulation of \eqref{eq:ECi} reads: Find $(E,H)\in\rgad\times\rgan$ such that
\begin{equation} \label{eq:ECiMixed}
\rot H+i\omega\sigma E=J,\quad\mu^{-1}\rot E=H \qquad \textrm{in } \om.
\end{equation}
Considering the norms we have
\begin{align*}
\norm{E}_{\r,\norm{\omega}\sigma,\mu^{-1}}^2
&=\norm{E}_{\lt,\norm{\omega}\sigma}^2
+\norm{\rot E}_{\lt,\mu^{-1}}^2,\\
\norm{H}_{\r,(\norm{\omega}\sigma)^{-1},\mu}^2
&=\norm{H}_{\lt,\mu}^2
+\norm{\rot H}_{\lt,(\norm{\omega}\sigma)^{-1}}^2,\\
\ttnorm{(E,H)}^2
&=\norm{E}_{\r,\norm{\omega}\sigma,\mu^{-1}}^2
+\norm{H}_{\r,(\norm{\omega}\sigma)^{-1},\mu}^2.
\end{align*}
Now Theorem \ref{thm:Gmain2} reads:
\begin{theo} 
\label{thm:ECi}
Let $(E,H),(\Et,\Ht)\in\rgad\times\rgan$ 
be the exact solution of \eqref{eq:ECiMixed} 
and any approximation, respectively. Then
$$\frac{\sqrt{2}}{\sqrt{2}+1} \Meci(\Et,\Ht)
\leq
\ttnorm{(E,H)-(\Et,\Ht)}^2
\leq
\frac{\sqrt{2}}{\sqrt{2}-1} \Meci(\Et,\Ht)$$
and
$$\frac{\sqrt{2}}{2(\sqrt{2}+1)} \cdot \frac{\Meci(\Et,\Ht)}{\norm{J}_{\lt,(\norm{\omega}\sigma)^{-1}}^2}
\leq
\frac{\ttnorm{(E,H)-(\Et,\Ht)}^2}{\ttnorm{(E,H)}^2}
\leq
\frac{\sqrt{2}}{\sqrt{2}-1} \cdot \frac{\Meci(\Et,\Ht)}{\norm{J}_{\lt,(\norm{\omega}\sigma)^{-1}}^2}$$
hold, where $\Meci(\Et,\Ht)=\norm{J-i\omega\sigma\Et-\rot\Ht}_{\lt,(\norm{\omega}\sigma)^{-1}}^2
+\norm{\Ht-\mu^{-1}\rot\Et}_{\lt,\mu}^2$.
\end{theo}


\subsection{Electro-Magnetic Problems (2D)} \label{subsec:EC2D}

Let $d=2$. In the following we simply indicate the changes compared to the previous section.
First, we have to understand the double $\rot$ as $\na^\bot\rot$, where
$$\rot E:=\div {\rm Q} E=\p_1E_2-\p_2E_1,\quad
\na^\bot H:={\rm Q}\na H=\begin{bmatrix}\p_2H\\-\p_1H\end{bmatrix},\quad
{\rm Q}:=\begin{bmatrix}0&1\\-1&0\end{bmatrix},$$
and $E\in\r$ is a vector field and $H\in\ho$ a scalar function. 
In the literature, the operator $\na^\bot$ is often called 
co-gradient or vector rotation $\vec{\rot}$ as well. 
Also $\mu$ is scalar.
The following table shows the relation to the notation of Section \ref{sec:G}:
\begin{center}\begin{tabular}{c|c||c|c||c|c||c|c}
$\alphao$ & $\alphat$ & $\A$ & $\As$ & $\hio$ & $\hit$ & $D(\A)$ & $D(\As)$\\
\hline
$\eps,\sigma$ & $\mu^{-1}$ & $\rot$ & $\na^{\bot}$ & $\lt$ & $\lt$ & $\rgad$ & $\hogan$
\end{tabular}\end{center}
The relation \eqref{partint} reads now
$$\forall\,\Phi\in\rgad\quad\forall\,\psi\in\hogan\qquad
\scplt{\rot\Phi}{\psi}=\scplt{\Phi}{\na^{\bot}\psi}.$$

\subsubsection*{Case I: a Maxwell Type Problem}

Now \eqref{eq:EC} reads: Find the electric field $E\in\r$ such that
\begin{align*}
\na^{\bot}\mu^{-1}\rot E+\eps E&=J&
\textrm{in }&\om,\\
n\times E&=0&
\textrm{on }&\gad,\\
\mu^{-1}\rot E&=0&
\textrm{on }&\gan.
\end{align*}
The mixed formulation of the problem is: Find $(E,H)\in\rgad\times\hogan$ such that
\begin{equation} \label{eq:EC2DMixed}
\na^{\bot}H+\eps E=J,\quad\mu^{-1}\rot E=H \qquad \textrm{in } \om.
\end{equation}
The norm for $H$ is
$$\norm{H}_{\ho,\eps^{-1},\mu}^2
=\norm{H}_{\lt,\mu}^2
+\norm{\na^{\bot}H}_{\lt,\eps^{-1}}^2.$$
Now Theorem \ref{thm:EC} (and thus Theorem \ref{thm:Gmain}) reads:
\begin{theo} 
\label{thm:EC2D}
Let $(E,H),(\Et,\Ht)\in\rgad\times\hogan$ 
be the exact solution of \eqref{eq:EC2DMixed} 
and any approximation, respectively. Then
$$\tnorm{(E,H)-(\Et,\Ht)}^2
=\Mec(\Et,\Ht),\quad
\frac{\tnorm{(E,H)-(\Et,\Ht)}^2}{\tnorm{(E,H)}^2}
=\frac{\Mec(\Et,\Ht)}{\norm{J}_{\lt,\eps^{-1}}^2}$$
hold, where $\Mec(\Et,\Ht)=\norm{J-\eps\Et-\na^{\bot}\Ht}_{\lt,\eps^{-1}}^2
+\norm{\Ht-\mu^{-1}\rot\Et}_{\lt,\mu}^2$.
\end{theo} 

\subsubsection*{Case II: Eddy-Current}

Now \eqref{eq:ECi} reads: Find the electric field $E\in\r$ such that
\begin{align*}
\na^{\bot}\mu^{-1}\rot E+i\omega\sigma E&=J&
\textrm{in }&\om,\\
n\times E&=0&
\textrm{on }&\gad,\\
\mu^{-1}\rot E&=0&
\textrm{on }&\gan.
\end{align*}
The mixed formulation of the problem is: Find $(E,H)\in\rgad\times\hogan$ such that
\begin{equation} \label{eq:ECi2DMixed}
\na^{\bot}H+i\omega\sigma E=J,\quad\mu^{-1}\rot E=H \qquad \textrm{in } \om.
\end{equation}
The norm for $H$ is
$$\norm{H}_{\ho,(\norm{\omega}\sigma)^{-1},\mu}^2
=\norm{H}_{\lt,\mu}^2
+\norm{\na^{\bot}H}_{\lt,(\norm{\omega}\sigma)^{-1}}^2.$$
Now Theorem \ref{thm:ECi} (and thus Theorem \ref{thm:Gmain2}) reads:
\begin{theo} 
\label{thm:ECi2D}
Let $(E,H),(\Et,\Ht)\in\rgad\times\hogan$ 
be the exact solution of \eqref{eq:ECi2DMixed}
and any approximation, respectively. Then
$$\frac{\sqrt{2}}{\sqrt{2}+1} \Meci(\Et,\Ht)
\leq
\ttnorm{(E,H)-(\Et,\Ht)}^2
\leq
\frac{\sqrt{2}}{\sqrt{2}-1} \Meci(\Et,\Ht)$$
and
$$\frac{\sqrt{2}}{2(\sqrt{2}+1)} \cdot \frac{\Meci(\Et,\Ht)}{\norm{J}_{\lt,(\norm{\omega}\sigma)^{-1}}^2}
\leq
\frac{\ttnorm{(E,H)-(\Et,\Ht)}^2}{\ttnorm{(E,H)}^2}
\leq
\frac{\sqrt{2}}{\sqrt{2}-1} \cdot \frac{\Meci(\Et,\Ht)}{\norm{J}_{\lt,(\norm{\omega}\sigma)^{-1}}^2}$$
hold, where $\Meci(\Et,\Ht)=\norm{J-i\omega\sigma\Et-\na^{\bot}\Ht}_{\lt,(\norm{\omega}\sigma)^{-1}}^2
+\norm{\Ht-\mu^{-1}\rot\Et}_{\lt,\mu}^2$.
\end{theo}


\subsection{Linear Elasticity} \label{subsec:LE}

We will need $\nas$, which is the symmetric part of the gradient\footnote{Here, 
as usual in elasticity the gradient $\na u$
is to be understood as the Jacobian of the vector field $u$.}
$$\nas u:=\sym\na u=\frac{1}{2}\big(\na u+(\na u)^\top\big),$$
where ${}^\top$ denotes the transpose. 
$\nas u$, often denoted by $\eps(u)$, is also called the infinitesimal strain tensor.
For a tensor $\sigma$ the notation $\sigma\in\d$ and the application of $\Div$ to $\sigma$ 
is to be understood row-wise as the usual divergence $\div$. Moreover, we define
$$\Divs\sigma:=\Div\sym\sigma.$$
The following table shows the relation to the notation of Section \ref{sec:G}.
\begin{center}\begin{tabular}{c|c||c|c||c|c||c|c}
$\alphao$ & $\alphat$ & $\A$ & $\As$ & $\hio$ & $\hit$ & $D(\A)$ & $D(\As)$\\
\hline
$\rho$ & $\L$ & $\nas$ & $-\Divs$ & $\lt$ & $\lt$ & $\hogad$ & $\sym^{-1}\dgan$
\end{tabular}\end{center}
The notation $\sigma\in\sym^{-1}\dgan$ means $\sym\sigma\in\dgan$.
More precisely, $\psi\in D(\As)$ if and only if
$$\forall\,\varphi\in D(\A)=\hogad\qquad
\scplt{\nas\varphi}{\psi}=\scplt{\varphi}{\As\psi}.$$
Since $\scplt{\nas\varphi}{\psi}=\scplt{\na\varphi}{\sym\psi}$
we see that this holds if and only if $\sym\psi\in\dgan$ and $\As\psi=-\Div\sym\psi$.
Equation \eqref{partint} turns into
$$\forall\,\varphi\in\hogad\quad\forall\,\psi\in\sym^{-1}\dgan
\qquad\scplt{\nas\varphi}{\psi}=-\scplt{\varphi}{\Divs\psi}.$$

\subsubsection*{Case I}

Find the displacement vector field $u\in\ho$ such that
\begin{align}
-\Div\L\nas u+\rho\,u&=f&
\textrm{in }&\om,\nonumber\\
u&=0&
\textrm{on }&\gad,\label{eq:LE}\\
\L\nas u \cdot n&=0&
\textrm{on }&\gan.\nonumber
\end{align}
The fourth order stiffness tensor of elastic moduli $\L\in\li$,
mapping symmetric matrices to symmetric matrices point-wise,
and the second order tensor (quadratic matrix) of reaction $\rho$ 
are assumed to be symmetric, real valued, and uniformly positive definite.
The vector field $f$ (body force) belongs to $\lt$ and
the dual variable for this problem is the Cauchy stress tensor $\sigma=\L\nas u\in\d$.
Note that $\sigma$ is indeed symmetric.
We note that the first equation in \eqref{eq:LE} can also be written as
$$-\Divs\L\nas u+\rho\,u=f.$$
The mixed formulation of \eqref{eq:LE} reads: Find $(u,\sigma)\in\hogad\times\dgan$ such that
\begin{equation} \label{eq:LEMixed}
-\Div\sigma+\rho\,u=f,\quad\L\nas u=\sigma \qquad \textrm{in } \om.
\end{equation}
For the norms we have
\begin{align*}
\norm{u}_{\ho,\rho,\L}^2
&=\norm{u}_{\lt,\rho}^2
+\norm{\nas u}_{\lt,\L}^2,\\
\norm{\sigma}_{\sym^{-1}\d,\rho^{-1},\L^{-1}}^2
&=\norm{\sigma}_{\lt,\L^{-1}}^2
+\norm{\Divs\sigma}_{\lt,\rho^{-1}}^2,\\
\tnorm{(u,\sigma)}^2
&=\norm{u}_{\ho,\rho,\L}^2
+\norm{\sigma}_{\sym^{-1}\d,\rho^{-1},\L^{-1}}^2.
\end{align*}
Now Theorem \ref{thm:Gmain} reads:

\begin{theo} 
\label{thm:LE}
Let $(u,\sigma),(\ut,\tilde{\sigma})\in\hogad\times\sym^{-1}\dgan$ 
be the exact solution of \eqref{eq:LEMixed}
and any approximation, respectively. Then
\begin{align*}
\tnorm{(u,\sigma)-(\ut,\tilde{\sigma})}^2
=\Mle(\ut,\tilde{\sigma}),\quad
\frac{\tnorm{(u,\sigma)-(\ut,\tilde{\sigma})}^2}{\tnorm{(u,\sigma)}^2}
=\frac{\Mle(\ut,\tilde{\sigma})}{\norm{f}_{\lt,\rho^{-1}}^2}
\end{align*}
hold, where $\Mle(\ut,\tilde{\sigma})=\norm{f-\rho\,\ut+\Divs\tilde{\sigma}}_{\lt,\rho^{-1}}^2
+\norm{\tilde{\sigma}-\L\nas\ut}_{\lt,\L^{-1}}^2$.
Moreover, since the tensor $\sigma$ is symmetric the above results hold for all pairs
$(\ut,\tilde{\sigma})\in\hogad\times\dgan$ with symmetric tensor $\tilde{\sigma}$,
and the functional simplifies to
$\Mle(\ut,\tilde{\sigma})=\norm{f-\rho\,\ut+\Div\tilde{\sigma}}_{\lt,\rho^{-1}}^2
+\norm{\tilde{\sigma}-\L\nas\ut}_{\lt,\L^{-1}}^2$.
\end{theo}

\subsubsection*{Case II}

Find the displacement vector field $u\in\ho$ such that
\begin{align}
-\Div\L\nas u+i\omega\rho\,u&=f&
\textrm{in }&\om,\nonumber\\
u&=0&
\textrm{on }&\gad,\label{eq:LEi}\\
\L\nas u \cdot n&=0&
\textrm{on }&\gan,\nonumber
\end{align}
where $\L,\rho$, and $f$ are as before, and $\omega \in \realsnotz$.
The dual variable for this problem is the Cauchy stress tensor $\sigma=\L\nas u\in\d$.
We note again that $\sigma$ is symmetric,
and that the first equation of \eqref{eq:LEi} can also be written as
$$-\Divs\L\nas u+i\omega\rho\,u=f.$$
The mixed formulation of \eqref{eq:LEi} reads: Find $(u,\sigma)\in\hogad\times\dgan$ such that
\begin{equation} \label{eq:LEiMixed}
-\Div\sigma+i\omega\rho\,u=f,\quad\L\nas u=\sigma \qquad \textrm{in } \om.
\end{equation}
For the norms we have
\begin{align*}
\norm{u}_{\ho,\norm{\omega}\rho,\L}^2
&=\norm{u}_{\lt,\norm{\omega}\rho}^2
+\norm{\nas u}_{\lt,\L}^2,\\
\norm{\sigma}_{\sym^{-1}\d,(\norm{\omega}\rho)^{-1},\L^{-1}}^2
&=\norm{\sigma}_{\lt,\L^{-1}}^2
+\norm{\Divs\sigma}_{\lt,(\norm{\omega}\rho)^{-1}}^2,\\
\ttnorm{(u,\sigma)}^2
&=\norm{u}_{\ho,\norm{\omega}\rho,\L}^2
+\norm{\sigma}_{\sym^{-1}\d,(\norm{\omega}\rho)^{-1},\L^{-1}}^2.
\end{align*}
Now Theorem \ref{thm:Gmain2} reads:

\begin{theo} 
\label{thm:LEi}
Let $(u,\sigma),(\ut,\tilde{\sigma})\in\hogad\times\sym^{-1}\dgan$ 
be the exact solution of \eqref{eq:LEiMixed}
and any approximation, respectively. Then
$$\frac{\sqrt{2}}{\sqrt{2}+1} \Mlei(\ut,\tilde{\sigma})
\leq
\ttnorm{(u,\sigma)-(\ut,\tilde{\sigma})}^2
\leq
\frac{\sqrt{2}}{\sqrt{2}-1} \Mlei(\ut,\tilde{\sigma})$$
and
$$\frac{\sqrt{2}}{2(\sqrt{2}+1)} \cdot \frac{\Mlei(\ut,\tilde{\sigma})}{\norm{f}_{\lt,(\norm{\omega}\rho)^{-1}}^2}
\leq
\frac{\ttnorm{(u,\sigma)-(\ut,\tilde{\sigma})}^2}{\ttnorm{(u,\sigma)}^2}
\leq
\frac{\sqrt{2}}{\sqrt{2}-1} \cdot \frac{\Mlei(\ut,\tilde{\sigma})}{\norm{f}_{\lt,(\norm{\omega}\rho)^{-1}}^2}$$
hold, where $\Mlei(\ut,\tilde{\sigma})=\norm{f-i\omega\rho\,\ut+\Divs\tilde{\sigma}}_{\lt,(\norm{\omega}\rho)^{-1}}^2
+\norm{\tilde{\sigma}-\L\nas\ut}_{\lt,\L^{-1}}^2$.
Moreover, since the tensor $\sigma$ is symmetric the above results hold for all pairs
$(\ut,\tilde{\sigma})\in\hogad\times\dgan$ with symmetric $\tilde{\sigma}$,
and the functional simplifies to
$\Mlei(\ut,\tilde{\sigma})=\norm{f-i\omega\rho\,\ut+\Div\tilde{\sigma}}_{\lt,(\norm{\omega}\rho)^{-1}}^2
+\norm{\tilde{\sigma}-\L\nas\ut}_{\lt,\L^{-1}}^2$.
\end{theo}


\subsection{Different Boundary Conditions and Other Problems}

We note that the (non-normalized) error equalities 
and error estimates hold without change with non-homogenous boundary conditions. 
Also Robin boundary conditions can be treated (see Appendix \ref{app:inhomobc}).

It is clear that the list of applications of our theory is much longer. For example:
\begin{itemize}
\item 
generalized reaction-diffusion, linear accoustics and electromagnetics 
on Riemannian manifolds\footnote{Here $\ed$ and $\delta$ denote the exterior and co-derivative, respectively.}
$$-\delta\ed+1,\quad
-\delta\ed+i$$
\item 
the fourth order problem 
$$\div\Div\na\na+1,\quad
\div\Div\na\na+i$$
\item 
the biharmonic problem 
$$\Delta\Delta+1,\quad
\Delta\Delta+i$$
\item 
certain generalized Stokes and Oseen type problems
\end{itemize}



\bibliographystyle{plain} 
\bibliography{/Users/paule/Library/texmf/tex/TeXinput/bibtex/paule}


\appendix

\section{Inhomogeneous and More Boundary Conditions}
\label{app:inhomobc}

We will demonstrate that our results also hold 
for Robin type boundary conditions, which means that our results are true
for many commonly used boundary conditions. Moreover, we emphasize that we can also handle
inhomogeneous boundary conditions.
Since it is clear that this method works in the general setting for both Cases I and II, 
we will demonstrate it here just for a simple reaction-diffusion type model problem
belonging to the class of Case I.
Let $\om$ be as in the latter section and now the boundary $\ga$ 
be decomposed into three disjoint parts $\gad$, $\gan$ and $\gar$.

The model problem is:
Find the scalar potential $u\in\ho$ such that
\begin{align*}
-\div\na u+u&=f&\textrm{in }&\om,\\
u&=g_1&\textrm{on }&\gad,\\
n\cdot\na u&=g_2&\textrm{on }&\gan,\\
n\cdot\na u+\gamma u&=g_3&\textrm{on }&\gar
\end{align*}
hold. Hence, on $\gad, \gan$ and $\gar$ we impose Dirichlet, Neumann 
and Robin type boundary conditions, respectively. 
In the Robin boundary condition, we assume that the coefficient $\gamma\ge\gamma_0>0$ 
belongs to $\li$. The dual variable for this problem 
is the flux $p:=\na u\in\d$.
Furthermore, as long as $\gar\neq\emptyset$ and
to avoid tricky discussions about traces 
and the corresponding $\hmoh$-spaces of $\ga$, $\gad, \gan$, and $\gar$,
which can be quite complicated, we assume for simplicity that
$u\in\htwo$. Then, $p\in\ho$ and all $g_{i}$ belong to $\lt$ even to $\hoh$ of $\ga$.
For the norms we simply have
$$\tnorm{(u,p)}^2
=\normho{u}^2
+\normd{p}^2.$$

\begin{theo}
\label{thm:dnr}
For any approximation pair $(\ut,\pt)\in\htwo\times\ho$ with
$u-\ut\in\hogad$ and $p-\pt\in\dgan$ as well as
$n\cdot(p-\pt)+\gamma(u-\ut)=0$ on $\gar$
$$\tnorm{(u,p)-(\ut,\pt)}^2
+2\norm{u-\ut}_{\lt(\gar),\gamma}^2
=\M(\ut,\pt)$$
holds with $\M(\ut,\pt) := \normlt{f-\ut+\div \pt}^2 + \normlt{\pt-\na \ut}^2$.
Moreover, $\norm{u-\ut}_{\lt(\gar),\gamma}=\norm{n\cdot(p-\pt)}_{\lt(\gar),\gamma^{-1}}$.
\end{theo}

\begin{proof}
Following the proof of Theorem \ref{thm:Gmain} we have
\begin{align*}
\M(\ut,\pt)
&=\ubr{\normho{u-\ut}^2
+\normd{p-\pt}^2}_{\ds=\tnorm{(u,p)-(\ut,\pt)}^2}
+2\Re\scplt{u-\ut}{\div(\pt-p)}
+2\Re\scplt{\na(u-\ut)}{\pt-p}.
\end{align*}
Moreover, since $n\cdot(\pt-p)$ and $u-\ut$ belong to $\lt(\ga)$ we have 
\begin{align}
&\scplt{\na(u-\ut)}{\pt-p}
+\scplt{u-\ut}{\div(\pt-p)} \nonumber \\
&=\scp{n\cdot(\pt-p)}{u-\ut}_{\lt(\ga)}
=\scp{n\cdot(\pt-p)}{u-\ut}_{\lt(\gar)}
=\scp{\gamma(u-\ut)}{u-\ut}_{\lt(\gar)}. \label{eq:partintrobin}
\end{align}
As $\scp{\gamma(u-\ut)}{u-\ut}_{\lt(\gar)}=\scp{\gamma^{-1}n\cdot(p-\pt)}{n\cdot(p-\pt)}_{\lt(\gar)}$
we get the assertion.
\end{proof}

\begin{rem}
If all $g_{i}=0$, we can set $(\ut,\pt)=(0,0)$ and get
$$\tnorm{(u,p)}^2
+2\norm{u}_{\lt(\gar),\gamma}^2
=\normlt{f}^2,$$
which follows also by
\begin{align*}
\normlt{f}^2
&=\normlt{\div p}^2
+\normlt{u}^2
-2\Re\scplt{\div\na u}{u}\\
&=\normlt{\div p}^2
+\normlt{u}^2
+2\normlt{\na u}
-2\Re\scp{n\cdot\na u}{u}_{\lt(\ga)}\\
&=\normlt{\div p}^2
+\normlt{u}^2
+2\normlt{\na u}
-2\Re\ubr{\scp{n\cdot\na u}{u}_{\lt(\gar)}}_{\ds=-\norm{u}_{\lt(\gar),\gamma}^2}.
\end{align*}
Thus, in this case the assertion of Theorem \ref{thm:dnr} has a normalized counterpart as well.
\end{rem}

If $\gar=\emptyset$ we have a pure mixed Dirichlet and Neumann boundary.

\begin{theo} 
\label{thm:dnrnorobin}
Let $\gar=\emptyset$.
For any approximation $(\ut,\pt)\in\ho\times\d$ with
$u-\ut\in\hogad$ and $p-\pt\in\dgan$ we have
$$\tnorm{(u,p)-(\ut,\pt)}^2
=\M(\ut,\pt).$$
\end{theo}

\begin{cor}
Let $\gar=\emptyset$.
Theorem \ref{thm:dnrnorobin} provides the well known a posteriori error estimates 
for the primal and dual problems. 
\begin{itemize}
\item[\bf(i)]
For any $\ut\in\ho$ with $u-\ut\in\hogad$ it holds
$\ds\normho{u-\ut}^2
=\min_{\substack{\psi\in\d\\p-\psi\in\dgan}}\M(\ut,\psi)
=\M(\ut,p)$.
\item[\bf(ii)]
For any $\pt\in\d$ with $p-\pt\in\dgan$ it holds
$\ds\normd{p-\pt}^2
=\min_{\substack{\varphi\in\ho\\u-\varphi\in\hogad}}\M(\varphi,\pt)
=\M(u,\pt)$.
\end{itemize}
\end{cor}

\end{document}